\documentclass[final,leqno]{siamltex}
\pdfoutput=1

\usepackage{amsmath}
\usepackage{epsfig}
\usepackage{graphicx}
\usepackage{amssymb}

\newtheorem{remark}{Remark}[section]

\def\a{\alpha}

\def\Lam{{\Lambda}}
\def\Ome{{\Omega}}

\def\Del{{\Delta}}
\def\nab{{\nabla}}

\def\p{{\partial}}
\def\reff#1{\eqref{#1}}
\def\norm#1#2{\Vert\,#1\,\Vert_{#2}}

\def\cQ{{\mathcal Q}}
\def\cR{{\mathcal R}}
\def\cS{{\mathcal S}}
\def\cT{{\mathcal T}}

\def\no{{\nonumber}}

\def\div{{\mbox{\rm div\,}}}

\def\p{{\partial}}

\def\tp{\widetilde{p}}

\def\nab{\nabla}
\def\Ome{\Omega}

\def\Del{\Delta}

\def\bE{\mathbf{E}}
\def\hE{\widehat{E}}
\def\bC{\mathbf{C}}
\def\bbf{\mathbf{f}}
\def\bu{\mathbf{u}}

\def\bv{\mathbf{v}}
\def\bw{\mathbf{w}}

\def\bH{\mathbf{H}}
\def\bV{\mathbf{V}}

\def\bP{\mathbf{P}}

\def\bV{\mathbf{V}}

\def\bX{\mathbf{X}}

\def\bLamu{\mathbf{\Lam}_{\bu}}
\def\bThetau{\mathbf{\Theta}_{\bu}}
\def\R{\mathbb{R}}

\def\T{{\mathcal T}_h}

\begin{document}


\title{Analysis of Fully Discrete Finite Element Methods for a System
of Differential Equations Modeling Swelling Dynamics of Polymer Gels } 

\author{
Xiaobing Feng\thanks{Department of Mathematics, The University of
Tennessee, Knoxville, TN 37996, U.S.A. ({\tt xfeng@math.utk.edu}). The
work of this author was partially supported by the NSF grant DMS-071083.}
\and
Yinnian He\thanks{College of Sciences, Xi'an Jiaotong University, Xi'an, 
Shaanxi 710049, P. R. China ({\tt heyn@mail.xjtu.edu.cn}). The work
of this author was partially supported by the NSF of China grant \#10671154 
and by the National Basic Research Program of China grant \#2005CB321703.}
}

\maketitle




\large

\begin{abstract}
The primary goal of this
paper is to develop and analyze some fully discrete finite element
methods for a displacement-pressure model modeling swelling dynamics
of polymer gels under mechanical constraints. In the model, the
swelling dynamics is governed by the solvent permeation and the
elastic interaction; the permeation is described by a pressure
equation for the solvent, and the elastic interaction is described
by displacement equations for the solid network of the gel. The
elasticity is of long range nature and gives effects for the solvent
diffusion. It is the fluid-solid interaction in the gel network
drives the system and makes the problem interesting and difficult.
By introducing an ``elastic pressure" (or ``volume change function")
we first present a reformulation of the original model, we then
propose a time-stepping scheme which decouples the PDE system
at each time step into two sub-problems, one of which is
a generalized Stokes problem for the displacement vector field
(of the solid network of the gel) and another is a diffusion
problem for a ``pseudo-pressure" field (of the solvent of the gel).
To make such a multiphysical approach feasible, it is vital to find
admissible constraints to resolve the uniqueness issue for
the generalized Stokes problem and to construct a
``good" boundary condition for the diffusion equation
so that it also becomes uniquely solvable. The key to the
first difficulty is to discover certain conservation laws (or conserved
quantities) for the PDE solution of the original model,
and the solution to the second difficulty is to use the
generalized Stokes problem to generate a boundary condition
for the diffusion problem. This then lays down the theoretical 
foundation for one to utilize any convergent Stokes solver (and its code) 
together with any convergent diffusion equation solver (and its code) 
to solve the polymer gel model. In the paper, the Taylor-Hood
mixed finite element method combined with the continuous 
linear finite element method are chosen as an example
to present the ideas and to demonstrate the viability of
the proposed multiphysical approach.
It is proved that, under a mesh constraint, both the proposed semi-discrete 
(in space) and fully discrete methods enjoy some discrete energy laws which 
mimic the differential energy law satisfied by the PDE solution. 
Optimal order error estimates in various norms are established for the 
numerical solutions of both the semi-discrete and fully discrete
methods. Numerical experiments are also presented to show the efficiency 
of the proposed approach and methods.
\end{abstract}



\pagestyle{myheadings} \thispagestyle{plain} \markboth{XIAOBING FENG
AND YINNIAN HE} {FINITE ELEMENT APPROXIMATIONS OF SWELLING DYNAMICS
OF GELS}

\section{Introduction}\label{sec-1}

A gel is a soft poroelastic material which consists of a solid network
and a colloidal solvent. The solid network spans the volume of the
solvent medium. The solvent can
permeate through the solid network and the permeation can be controlled by
external forces. Both by weight and volume, gels are mostly liquid in
composition and thus exhibit densities similar to liquids. However, they have the
structural coherence of a solid and can be deformed. A gel network can be composed
of a wide variety of materials, including particles, polymers and proteins,
which then gives different types gels such hydrogels, organogels and xerogels
(cf. \cite{doi90,hamley07}).  Gels have some fascinating properties, in particular,
they display thixotropy which means that they become fluid when agitated, but
resolidify when resting. In general, gels are apparently solid, jelly-like
materials, they exhibit an important state of matter found in a wide variety
of biomedical and chemical systems (cf. \cite{doi90,de86,yd04a,yd04b}
and the references therein).

This paper develops and analyzes some fully discrete finite element 
methods for a displacement-pressure model for polymer gels. The model, 
which was proposed by M. Doi {\em et al} in \cite{doi90,yd04a,yd04b}, 
describes swelling dynamics of polymer gels (under mechanical constraints).  
Let $\Ome \subset \R^d\, (d=1,2,3)$ be a
bounded domain and denote the initial region occupied by the gel.
Let $\bu(x,t)$ denote the displacement of the gel at the point
$x\in \Ome$ in the space and at the time $t$,  $\bv_s(x,t)$ and $p(x,t)$
be the velocity and the pressure of the solvent at $(x,t)$.
Following \cite{doi90}, the governing equation for the
swelling dynamics of polymer gels are given by
\begin{align} \label{e1.1}
\div\bigl( \sigma(\bu) -p I\bigr) &= 0, \\
\xi (\bv_s-\bu_t)&= -(1-\phi) \nab p, \label{e1.2}\\
\div\bigl(\phi\bu_t+(1-\phi) \bv_s \bigr) &=0. \label{e1.3}
\end{align}
Here $\xi$ is the friction constant associated with the
motion of the polymer relative to the solvent, $\phi$ is the volume
fraction of the polymer, $I$ denotes the $d\times d$ identity
matrix, and $\sigma(\bu)$ stands for the stress tensor of the gel
network, which is given by a constitutive equation. In this paper,
we use the following linearized form of the stress tensor:
\begin{equation}\label{e1.4}
\sigma(\bu):=\bigl(K-\frac23G\bigr) \div \bu \,I + 2G\, \varepsilon(\bu),\qquad
\varepsilon(\bu):= \frac12 \bigl( \nab \bu +\nab\bu^T \bigr),
\end{equation}
where $K$ and $G$ are respectively the bulk and shear modulus of the gel
(cf. \cite{doi90,coussy04}).
We remark that \eqref{e1.1} stands for the force balance, \eqref{e1.2}
states Darcy's law for the permeation of solvent through the gel network,
and \eqref{e1.3} describes the incompressibility condition. In addition,
if we introduce the total stress $\widetilde{\sigma}(\bu, p):=\sigma(\bu)-p I$,
then equation \eqref{e1.1} becomes $\div \widetilde{\sigma}(\bu,p)=0$.

Substituting \eqref{e1.4} into \eqref{e1.1} and \eqref{e1.2}
into \eqref{e1.3} yield the following basic equations for swelling dynamics of
polymer gels (see \cite{yd04a})
\begin{alignat}{2} \label{e1.5}
\alpha \nab\div\bu + \beta \Del \bu &= \nab p,
&&\qquad \alpha:= K+\frac{G}3, \quad \beta:=G, \\
\div \bu_t &= \kappa \Del p, &&\qquad \kappa:= \frac{(1-\phi)^2}{\xi}, \label{e1.6}
\end{alignat}
which hold in the space-time domain $\Ome_T:=\Ome\times (0,T)$ for some given $T>0$.

To close the above system, we need to prescribe boundary and initial conditions.
Only one initial condition is required for the system, which is
\begin{equation}\label{e1.7}
\bu(\cdot,0)=\bu_0(\cdot) \qquad\mbox{in } \Ome.
\end{equation}
Various sets of boundary conditions are possible and each of them describes
a certain type mechanical condition and solvent permeation condition
(cf. \cite{yd04a,yd04b}). In this paper we consider the following
set of boundary conditions
\begin{align} \label{e1.8}
(\sigma(\bu)- p I)\nu =\bbf, \qquad
\frac{\p p}{\p \nu} = 0 \quad \mbox{on } \Ome_T:=\p\Ome\times (0,T),
\end{align}
where $\nu$ denotes the outward normal to $\p\Ome$.
\eqref{e1.8}$_1$ means that the mechanical force $\bbf$ is
applied on the boundary of the gel. Since $\frac{\p p}{\p \nu}
= (\bv_s-\bu_t)\cdot \nu$, hence, \eqref{e1.8}$_2$
implies that the solvent can not permeate through the gel boundary.
We also remark that the force function $\bbf$ must satisfy the
compatibility condition
\[
\int_{\p\Ome} \bbf\, dS =0.
\]

Problem \eqref{e1.5}--\eqref{e1.8} is interesting and difficult due to
its multiphysical nature which describes the complicate fluid and solid
interaction inside the gel network. It is numerically tricky to solve
because it is difficult to design a good and workable time-stepping
scheme. For example, one natural attempt would be at each time step
first to solve a Poisson problem for $p$ and then to solve a linear
elasticity problem
for $\bu$. However, this strategy is difficult to realize because
there is no good way to compute the source term $\div \bu_t$ for
the Poisson equation. In fact, the strategy even has a difficulty to start
due to the fact that no initial condition is provided for the pressure $p$.

To overcome the difficulty, in this paper we shall use a reformulation of
system \eqref{e1.5}--\eqref{e1.6}, which is now introduced. Define
\begin{equation}\label{e1.9}
q:= \div \bu.
\end{equation}
Physically, $q$ measures the volume change of the solid network of the
gel, and often called ``elastic pressure" or
``volume change function". Taking divergence on \reff{e1.5} yields
\begin{equation*}
(\alpha+\beta) \Del q = \Del p,
\end{equation*}
which and \reff{e1.6} imply that $q$ satisfies the following diffusion equation
\begin{equation}\label{e1.10}
q_t =D \Del q, \qquad D:=\kappa \bigl(K+\frac43 G\bigr).
\end{equation}
However, the usefulness of the above diffusion equation is hampered by the
lack of boundary condition for $q$.
We like to note that the above diffusion equation for $q$ was first noticed
by M. Doi \cite{doi_private}, but it was not utilized before
exactly because of the lack of boundary condition for $q$.
Nevertheless, using the new variable $q$ we can rewrite
\eqref{e1.5}--\eqref{e1.6} as
\begin{alignat}{2} \label{e1.11}
\beta \Del \bu &= \nab \tp, &&\qquad \tp:=p-\alpha q, \\
\div \bu &=q, &&  \label{e1.12} \\
q_t &=\kappa \Del p, &&\qquad p=\tp +\alpha q. \label{e1.13}
\end{alignat}
An immediate consequence of the above reformulation is
that \eqref{e1.11}-\eqref{e1.12} implies $(\bu, \tp)$ satisfies
the generalized Stokes equations with $q$ being the source term
at each time $t$, and $q$ satisfies a diffusion equation and
it interacts with $(\bu, \tp)$ only at the boundary $\p\Ome$.

This is a key observation because it not only reveals the underlying
physical process of swelling dynamics of the gel, but also gives the
``right" hint on how the problem should be solved numerically.
This indeed motivates the main idea of this paper, that is, {\em at each
time step, we first solve the generalized Stokes problem for $(\bu,\tp)$,
which in turn provides (implicitly) an updated boundary condition for $q$,
we then use this new boundary condition to solve the diffusion equation for
$q$. The process is repeated iteratively until the final time step is reached}.
However, in order to make this idea work, there is one crucial
issue needs to be addressed. That is, for a given $q$ the generalized Stokes
problem for $(\bu,\tp)$ is only unique up to additive constants.
Clearly, how to correctly enforce the uniqueness of the generalized Stokes
problem is the bottleneck of this approach. It is easy to understand
that one can not use arbitrary constraints to fix $(\bu,\tp)$ because this
will lead to bad or even divergent numerical schemes if the exact PDE solution
does not satisfy the constraints. Instead, the constraints
which can be used to fix $(\bu,\tp)$ should be those satisfied by
the exact solution of the PDE system. To the end, we need to
discover some invariant (or conserved) quantities for the
exact PDE solution. It turns out that the situation is
precisely what we anticipated and wanted.  We are able to show that
the exact PDE solution $(\bu, p,\tp, q)$ satisfies
the following identities (see Section \ref{sec-2} below for a proof):
\begin{align}\label{e1.14}
\int_\Ome q(x,t) dx &\equiv C_q:= \int_\Ome q_0(x) dx := \int_\Ome \div \bu_0(x) dx,\\
\int_\Ome p(x,t) dx &\equiv C_p:=c_d C_q 
-\frac{1}d \int_{\p\Ome} \bbf(x,t) \cdot x \, dS, \label{e1.15} \\
\int_{\p\Ome} \bu(x,t) \cdot \nu \, dS &\equiv C_{\bu}
:= \int_{\p\Ome} \bu_0(x) \cdot \nu \, dS= C_{q},
\label{e1.16}
\end{align}
where $d$ denote the dimension of $\Ome$ and
\begin{equation}\label{e1.17}
c_d:= \alpha + \frac{\beta}{d}
    =\begin{cases}
      K+\frac{5G}6 &\quad\mbox{if } d=2,\\
      K+\frac{2G}3 &\quad\mbox{if } d=3.
     \end{cases}
\end{equation}
Obviously, the right-hand sides of \eqref{e1.14} and \eqref{e1.16} are
constants. The right-hand side of \eqref{e1.15} is also a constant provided that
$\bbf$ is independent of $t$, otherwise, it is a known function of $t$.  In this
paper, we shall only consider the case that $\bbf$ is independent of $t$.
It follows from \eqref{e1.14} and \eqref{e1.15} that
\begin{equation}\label{e1.18}
\int_\Ome \tp(x,t) dx \equiv C_{\tp}:=C_p -\alpha C_q =\frac{\beta C_q}{d}
-\frac{1}d \int_{\p\Ome} \bbf(x) \cdot x \, dS.
\end{equation}
It is clear now that \eqref{e1.18} and \eqref{e1.16} provide two
natural conditions which can be used to uniquely determine the solution $(\bu, \tp)$
to the generalized Stokes problem \eqref{e1.11}--\eqref{e1.12} for
a given source term $q$.  This then leads to the following time-discretization
for problem \eqref{e1.5}--\eqref{e1.8}:

{\bf Algorithm 1:}

\begin{itemize}
\item[(i)]  Set $q^0=q_0:=\div \bu_0$ and $\bu^0:=\bu_0$.

\item[(ii)] For $n=0,1,2, \cdots$,  do the following two steps

{\em Step 1:} Solve for $(\bu^{n+1},\tp^{n+1})$ such that
\begin{alignat}{2} \label{e1.19}
-\beta \Del \bu^{n+1} + \nab \tp^{n+1} &=0,\quad
&&\qquad\mbox{in } \Ome_T, \\
\div \bu^{n+1} &=q^n &&\qquad\mbox{in } \Ome_T,  \label{e1.20} \\
\beta\frac{\p \bu^{n+1}}{\p\nu}-\tp^{n+1} \nu &= \bbf
&&\qquad\mbox{on } \p\Ome_T,  \label{e1.21} \\
(\tp^{n+1}, 1) =C_{\tp}, \quad \langle \bu^{n+1}, \nu\rangle &= C_{\bu}. &&\label{e1.22}
\end{alignat}

{\em Step 2:} Solve for $q^{n+1}$ such that
\begin{alignat}{2} \label{e1.23}
d_t q^{n+1} - \kappa \Del (\alpha q^{n+1}+\tp^{n+1}) &=0,
&&\qquad\mbox{in } \Ome_T,   \\
\alpha\frac{\p q^{n+1}}{\p \nu} &= -\frac{\p \tp^{n+1}}{\p \nu}
&&\qquad\mbox{on } \p\Ome_T,  \label{e1.24} \\
(q^{n+1}, 1)&=C_q, \label{e1.24a}
\end{alignat}
where $d_t q^{n+1}:= (q^{n+1}-q^n)/{\Del t}$.
\end{itemize}
We note that \eqref{e1.23} is the implicit Euler scheme, which 
is chosen just for the ease of presentation, it can be replaced by
other time-stepping schemes. \eqref{e1.24} provides a Neumann 
boundary condition for $q^{n+1}$. Another subtle issue is the
role which the initial value $\bu_0$ plays in the algorithm. 
Seems $\bu_0$ is only needed to produce $q_0$ and there is no need 
to have $\bu^0$ in order to execute the algorithm. However, to 
ensure the stability and convergence of the algorithm, it 
turns out that $\bu^0$ not only needs to be provided but also
must be carefully constructed when the algorithm is discretized 
(see Sections \ref{sec-3} and \ref{sec-4}).

The above algorithm has a couple attractive features.
First, it is easy to use. Second, it allows one to make
use of any available numerical methods (finite element,
finite difference, finite volume, spectral and discontinuous Galerkin)
and computer codes for the Stokes problem and the Poisson problem to
solve the gel swelling dynamics model \eqref{e1.5}--\eqref{e1.8}.
We remark that an almost same model as \eqref{e1.5}--\eqref{e1.8} also 
arise from different applications in poroelasticity and soil mechanics and
are known as Boit's consolidation model (cf. \cite{biot,murad} and 
the references therein). \cite{murad} proposed and analyzed a standard
finite element method which directly approximates $(\bu,p)$ under  
the (restrictive) divergence-free assumption on the initial condition $\bu_0$.


This paper consists of four additional sections. In Section \ref{sec-2}, we
first introduce notation used in this paper. We then
present a PDE analysis for the gel swelling dynamics
model \eqref{e1.5}--\eqref{e1.8},
which includes deriving a dissipative energy law, establishing existence and
uniqueness, and in particular, proving the conservation laws stated in
\eqref{e1.14}--\eqref{e1.16} and \eqref{e1.18}. In Section \ref{sec-3},
we first propose a semi-discrete (in space) finite element discretization 
for problem \eqref{e1.5}--\eqref{e1.8} based on the multiphysical 
reformulation \eqref{e1.11}--\eqref{e1.13}. The well-known 
Taylor-Hood mixed element and the $P_1$ conforming finite element 
are used as an example to present the ideas. It is proved that
the solution of the semi-discrete method satisfies a 
discrete energy law which mimics the differential energy
law enjoyed by the PDE solution, and the semi-discrete numerical 
solution also satisfies the conservation laws 
\eqref{e1.14}--\eqref{e1.16} and \eqref{e1.18}. 
We then derive optimal order error estimates in various norms
for the semi-discrete numerical solution. 
In Section \ref{sec-4}, fully discrete finite element
methods are constructed by combining the time-stepping 
scheme of Algorithm $1$ and the semi-discrete finite 
element methods of Section \ref{sec-3}. The main results of 
this section include proving a fully discrete
energy law for the numerical solution and establishing
optimal order error estimates for the fully discrete finite element 
methods. Finally, in Section \ref{sec-5}, we present some numerical
experiments to gauge the efficiency of the proposed approach and methods.

\section{PDE analysis for problem \eqref{e1.5}--\eqref{e1.8}} \label{sec-2}
The standard Sobolev space notation is used in this paper,
we refer to \cite{bs08,cia,temam} for their precise definitions.
In particular, $(\cdot,\cdot)$ and $\langle \cdot,\cdot\rangle$
denote respectively the standard $L^2(\Ome)$ and $L^2(\p\Ome)$ inner products.
For any Banach space $B$, we let $\mathbf{B}=[B]^d$ and use ${\mathbf{B}}'$
to denote its dual space. In particular, we use $(\cdot,\cdot)_{\small\rm dual}$
and $\langle \cdot,\cdot \rangle_{\small\rm dual}$ to denote
the dual products on $(H^1(\Ome))' \times H^1(\Ome)$ and
$\bH^{-\frac12}(\p\Ome)\times \bH^{\frac12}(\p\Ome)$, respectively.
$\norm{\cdot}{L^p(B)}$ is a shorthand notation for
$\norm{\cdot}{L^p((0,T);B)}$.

We also introduce the function spaces
\begin{align*}
&L^2_0(\Omega):=\{q\in L^2(\Omega);\, (q,1)=0\}, \qquad
\bX:=\{\bv\in \bH^1(\Ome);\, \langle \bv,\nu\rangle=0\}.
\end{align*}
It is well known \cite{temam} that the following so-called
inf-sup condition holds in the space $\bX\times L^2_0(\Ome)$:
\begin{align}\label{e2.0}
\sup_{\bv\in \bX}\frac{(\div \bv,\varphi)}{\norm{\nab \bv}{L^2}}
\geq \alpha_0 \norm{\varphi}{L^2} \qquad \forall
\varphi\in L^2_0(\Ome),\quad \alpha_0>0.
\end{align}

Throughout the paper, we assume $\Omega \subset \R^d$ be a bounded
polygonal domain such that $\Delta: H^1_0(\Omega) \cap H^2(\Omega)
\rightarrow L^2(\Omega)$ is an isomorphism; see ~\cite{GT,gra}.
In addition, $C$ is used to denote a generic positive constant which is
independent of $\bu,\,p,\,\tp,\,q$, and the mesh parameters $h$ and $\Del t$.

We now give a definition of weak solutions to \reff{e1.5}-\reff{e1.8}.
\begin{definition}\label{weak1}
Let $(\bu_0, \bbf) \in \bH^1(\Ome) \times \bH^{-\frac12}(\p\Omega)$,
and $\langle \bbf, 1 \rangle_{\small\rm dual} =0$.
Given $T > 0$, a tuple $(\bu,p)$ with
\[
\bu\in L^\infty\bigl(0,T; \bH^1(\Ome)),\quad \div \bu \in H^1(0,T;H^{-1}(\Ome)), \quad
p\in L^2 \bigl(0,T; H^1(\Omega)\bigr),
\]
is called a weak solution to \reff{e1.5}--\reff{e1.8},
if there hold for almost every $t \in [0,T]$
\begin{alignat}{2}\label{e2.1}
\bigl((\div\bu)_t, \varphi \bigr)_{\small\rm dual}
+ \kappa \bigl(\nab p, \nab \varphi \bigr)
&= 0 &&\quad\forall \varphi \in H^1(\Ome), \\
\alpha \bigl( \div\bu, \div\bv \bigr) + \beta  \bigl( \nab\bu, \nab\bv \bigr)
- \bigl( p, \div \bv \bigr) &= \langle \bbf, \bv \rangle_{\small\rm dual}
&&\quad\forall \bv\in \bH^1(\Ome), \label{e2.2} \\
\bu(0) &= \bu_0.  && \label{e2.3}
\end{alignat}
\end{definition}

\medskip
Similarly, we define weak solutions to problem \reff{e1.11}-\reff{e1.13},
\reff{e1.7}-\reff{e1.8}.

\begin{definition}\label{weak2}
Let $(\bu_0, \bbf) \in \bH^1(\Ome) \times \bH^{-\frac12}(\p\Omega)$,
and $\langle \bbf, 1 \rangle_{\small\rm dual} =0$.
Given $T > 0$, a triple $(\bu,\tp, q)$ with
\begin{alignat*}{2}
&\bu\in L^\infty\bigl(0,T; \bH^1(\Ome)), &&\qquad
 \tp\in L^2 \bigl(0,T; L^2(\Omega)\bigr), \\
&q\in L^\infty(0,T;L^2(\Ome))\cap H^1\bigl(0,T; H^{-1}(\Omega)\bigr), &&\qquad
\alpha q+\tp \in L^2(0,T; H^1(\Ome)),
\end{alignat*}
is called a weak solution to \reff{e1.11}--\reff{e1.13}, \reff{e1.7}-\reff{e1.8}
if there hold for almost every $t \in [0,T]$
\begin{alignat}{2}\label{e2.4}
\beta \bigl( \nab\bu, \nab\bv \bigr)-\bigl( \tp, \div \bv \bigr)
&= \langle \bbf, \bv \rangle_{\small\rm dual} &&\quad\forall \bv\in \bH^1(\Ome), \\
\bigl(\div\bu, \varphi \bigr) &= \bigl(q, \varphi \bigr)
&&\quad\forall \varphi \in L^2(\Ome), \label{e2.5} \\
\bu(0) &= \bu_0, && \label{e2.5a} \\
\bigl(q_t, \psi \bigr)_{\small\rm dual}+\kappa \bigl(\nab (\alpha q+\tp),
\nab \psi \bigr) &=0
&&\quad\forall \psi \in H^1(\Ome) , \label{e2.6} \\
q(0)&= q_0:=\div \bu_0. && \label{e2.7}
\end{alignat}
\end{definition}

\begin{remark}
(a) Clearly, $p:=\alpha q +\tp$ gives back the pressure in the original
formulation. What interesting is that both $q$ and $\tp$ are only
$L^2$-functions in the spatial variable but their combination
$\alpha q+\tp$ is an $H^1$-function. In other words, the new
formulation provides an $L^2-L^2$ decomposition for the pressure $p$.
It turns out that this decomposition will have a significant numerical impact
because it allows one to use low order (hence cheap) finite elements
to approximate $q$ and $\tp$ but still to be able to approximate
the pressure $p$ with high accuracy. 

(b) \eqref{e2.6} implicitly imposes the following boundary condition for $q$:
\begin{equation}
\alpha\frac{\p q}{\p \nu} = -\frac{\p \tp}{\p \nu} \qquad\mbox{on }\p \Ome.
\end{equation}
\end{remark}

Since problem \eqref{e2.1}--\eqref{e2.3} consists of two linear equations,
its solvability should follows easily if we can establish a priori
energy estimates for its solutions.  The following dissipative
energy law just serves that purpose.

\medskip
\begin{lemma}\label{lem2.1}
Every weak solution $(\bu,p)$ of problem \eqref{e1.5}--\eqref{e1.8} satisfies
the following energy law:
\begin{align}\label{e2.9}
E(t) + \kappa \int_0^t \norm{\nab p(s)}{L^2}^2\, ds =E(0) \qquad\forall t\in (0,T],
\end{align}
where
\[
E(t):= \frac12 \Bigl[ \beta \norm{\nab \bu(t)}{L^2}^2
+\alpha \norm{\div \bu(t)}{L^2}^2
-2\langle \bbf, \bu(t) \rangle_{\small\rm dual} \Bigr].
\]
Moreover,
\begin{equation}\label{e2.9a}
\norm{(\div \bu)_t}{L^2(H^{-1})} = \kappa \norm{\Del p}{L^2(H^{-1})}
\leq E(0)^{\frac12}.
\end{equation}
\end{lemma}

\begin{proof}
We first consider the case $\bu\in H^1(0,T;\bH^1(\Ome))$.
Setting $\varphi=p$ in \eqref{e2.1} and $\bv=\bu_t$ in \eqref{e2.2} yield
\begin{align*}
\bigl(\div \bu_t(t),p(t) \bigr) +\kappa \norm{\nab p}{L^2}^2 &=0, \\
\frac{d}{dt} \Bigl[ \frac{\alpha}2 \norm{\div \bu(t)}{L^2}^2
+ \frac{\beta}2 \norm{\nab \bu(t)}{L^2}^2 \Bigr]
&= \frac{d}{dt} \langle \bbf, \bu(t) \rangle_{\small\rm dual}
  + \bigl(p(t), \div \bu_t(t) \bigr) .
\end{align*}
\eqref{e2.9} follows from adding the above two equations
and integrating the sum in $t$ over the interval $(0,s)$ for any
$s\in (0,T]$.

If $(\bu,p)$ is only a weak solution, then $\bu_t$ could not be used as
a test function in \reff{e2.2}. However, this difficulty can be
easily overcome by using $\bu^\delta_t$ as a test function
in \reff{e2.2}, where $\bu^\delta$ denotes a mollification
of $\bu$ through a symmetric mollifier (cf. \cite[Chapter 7]{GT}),
and by passing to the limit $\delta\to 0$. 

Finally, \reff{e2.9a} follows immediately from 
\reff{e2.1} and \reff{e2.9}. The proof is completed.
\end{proof}

\medskip
\begin{remark}
Lemma \ref{lem2.1} and Theorem \ref{thm2.1} can be easily carried over
to the reformulated problem \eqref{e1.11}--\eqref{e1.13},
\eqref{e1.7}--\eqref{e1.8}. The only difference is that the
energy law \reff{e2.9} now is replaced by the following
equivalent energy law:
\begin{align}\label{e2.10}
J(t)+\kappa \int_0^t \norm{\nab [\tp(s)+\alpha q(s)] }{L^2}^2\, ds
=J(0) \qquad\forall t\in (0,T],
\end{align}
and \reff{e2.9a} is replaced by
\begin{equation}\label{e2.10a}
\norm{q_t}{L^2(H^{-1})}
\leq \sqrt{\kappa}\norm{\nab (\tp+\alpha q)}{L^2(L^2)} \leq J(0)^{\frac12}.
\end{equation}
Where
\[
J(t):= \frac12 \Bigl[\beta \norm{\nab \bu(t)}{L^2}^2
+\alpha \norm{q(t)}{L^2}^2
-2\langle \bbf, \bu(t) \rangle_{\small\rm dual} \Bigr].
\]
We also note that the weak solution to problem \eqref{e1.11}--\eqref{e1.13},
\eqref{e1.7}--\eqref{e1.8} is understood in the sense of
Definition \ref{weak2}.
\end{remark}

\begin{lemma}\label{lem2.2}
Weak solutions of problem \eqref{e1.5}--\eqref{e1.8} and problem
\eqref{e1.11}--\eqref{e1.13},\eqref{e1.7}--\eqref{e1.8}
satisfy the conservation laws \reff{e1.14}--\reff{e1.16},
and \reff{e1.18}.
\end{lemma}

\begin{proof}
\reff{e1.14} and \reff{e1.16} follows immediately from taking
$\psi\equiv 1$ in \reff{e2.6}, $\varphi\equiv 1$ in both
\reff{e2.5} and \reff{e2.1} followed by integrating in $t$ and
appealing to the divergence theorem.

To prove \reff{e1.15}, taking $\bv=x$ in \reff{e2.2} and
using the identities $\nab x=I$ and $\div x =d$ we get
\[
\alpha \bigl( \div\bu, d\bigr) + \beta\bigl(\nab\bu, I\bigr)
- \bigl(p, d\bigr) =  \langle \bbf, x \rangle_{\small\rm dual}
\]
Hence,
\begin{align*}
\bigl(p, 1\bigr) &= \alpha \bigl(\div\bu, 1\bigr)
+\frac{\beta}{d}\bigl(\div\bu, 1\bigr)
- \frac{1}{d} \langle \bbf, x \rangle_{\small\rm dual} \\
&=\Bigl(\alpha +\frac{\beta}{d} \Bigr) \bigl(\div\bu_0, 1\bigr)
 -\frac{1}{d}\langle \bbf, x \rangle_{\small\rm dual} \\
&=c_d C_q -\frac{1}{d}\langle \bbf, x \rangle_{\small\rm dual},
\end{align*}
where $c_d$ and $C_q$ are defined in \reff{e1.17} and \reff{e1.14}.
So we obtain \reff{e1.15}.

Similarly, \reff{e1.18} follows by taking $\bv=x$ in \reff{e2.4}.
The proof is complete.
\end{proof}

\medskip
With the help of the above two lemmas, we can show the
solvability of problem \eqref{e1.5}--\eqref{e1.8}.

\begin{theorem}\label{thm2.1}
Suppose $(\bu_0, \bbf) \in \bH^1(\Ome) \times \bH^{-\frac12}(\p\Omega)$,
and $\langle \bbf, 1 \rangle_{\small\rm dual} =0$.  Then there exists a 
unique weak solution to \eqref{e1.5}--\eqref{e1.8} in the sense of 
Definition \ref{weak1}, and there exists a unique solution to 
to \reff{e1.11}--\reff{e1.13}, \reff{e1.7}--\reff{e1.8}.
\end{theorem}

\begin{proof}
The uniqueness is an immediate consequence of the energy laws \reff{e2.9} 
and \reff{e2.10} and the conservation laws \reff{e1.14}--\reff{e1.16} 
and \reff{e1.18}.

The existence can be easily proved by using (abstract) Galerkin method
and the standard compactness argument (cf. \cite{temam}). The
energy laws \reff{e2.9} and \reff{e2.10} provide the necessary uniform estimates
for the Galerkin approximate solutions. We omit the details
because the derivation is standard.
\end{proof}

Again, exploiting the linearity of the system, we have the following 
regularity results for the weak solution.

\begin{theorem}\label{regularity} 
Let $(\bu,p,\tp,q)$ be the weak solution of \reff{e2.4}--\reff{e2.7} 
Then there holds the following estimates: 
\begin{align}\label{e2.11}
\sqrt{\beta} \norm{\sqrt{t} \nab \bu_t}{L^2(L^2)}
+\sqrt{\alpha} \norm{\sqrt{t} q_t}{L^2(L^2)}
+\sqrt{\kappa} \norm{\sqrt{t} \nab p}{L^\infty(L^2)} 
&\leq J(0)^{\frac12}. \\
\sqrt{\beta} \norm{t \nab \bu_t}{L^\infty(L^2)}
+\sqrt{\alpha} \norm{t q_t}{L^\infty(L^2)}
+\sqrt{\kappa} \norm{t \nab p_t}{L^2(L^2)} 
&\leq [2J(0)]^{\frac12}. \label{e2.11a} \\
\norm{tq_{tt}}{L^2(H^{-1})} &\leq  [2J(0)]^{\frac12}.  \label{e2.11b} \\
\kappa \sqrt{\alpha} \norm{\sqrt{t} \Del p}{L^2(L^2)} 
=\sqrt{\alpha} \norm{\sqrt{t} q_t}{L^2(L^2)} 
&\leq J(0)^{\frac12}.\label{e2.11c} \\
\kappa \sqrt{\alpha} \norm{t \Del p}{L^\infty(L^2)} 
= \sqrt{\alpha} \norm{t q_t}{L^\infty(L^2)} 
&\leq [2J(0)]^{\frac12}.\label{e2.11d} \\
\norm{\nab\tp}{L^2(L^2)} + \norm{\nab\tp}{L^\infty(L^2)} 
+ \alpha \sqrt{\kappa} \norm{\nab q}{L^2(L^2)} &\quad \label{e2.11e} \\
+\alpha \sqrt{\kappa} \norm{\sqrt{t}\nab q}{L^\infty(L^2)} 
&\leq C(C_{\tp}, J(0)). \no \\
\beta \norm{\Del \bu}{L^\infty (L^2)} = \norm{\nab\tp}{L^\infty(L^2)} 
&\leq C(C_{\tp}, J(0)). \label{e2.11f} 
\end{align}
Where $C(T,C_{\tp}, J(0))$ denotes some positive constant which depends
on $T, C_{\tp}$ and  $J(0)$.
\end{theorem}

\begin{proof} 
Differentiating each of equations \reff{e2.4}, \reff{e2.5} and \reff{e2.6} 
with respect to $t$ yields (note that $\bbf$ is assumed to be independent of $t$)
\begin{alignat}{2}\label{e2.12}
\beta \bigl( \nab\bu_t, \nab\bv \bigr)-\bigl( \tp_t, \div \bv \bigr)
&= 0 &&\quad\forall \bv\in \bH^1(\Ome), \\
\bigl(\div\bu_t, \varphi \bigr) &= \bigl(q_t, \varphi \bigr)
&&\quad\forall \varphi \in L^2(\Ome), \label{e2.13} \\
\bigl(q_{tt}, \psi \bigr)_{\small\rm dual}+\kappa \bigl(\nab (\alpha q+\tp)_t,
\nab \psi \bigr) &=0
&&\quad\forall \psi \in H^1(\Ome). \label{e2.14}
\end{alignat}
Taking $\bv=t \bu_t$ in \reff{e2.12}, $\varphi=t \tp_t$ in \reff{e2.13},
$\psi=t p_t=t(\tp+\alpha q)_t$ in \reff{e2.6}, and adding the resulted 
equations we get
\begin{align*}
t\beta\norm{\nab \bu_t}{L^2}^2 +t\alpha \norm{q_t}{L^2}^2
+\frac{\kappa}2 \frac{d}{dt} \Bigl(t\norm{\nab p}{L^2}^2\Bigr)
= \frac{\kappa}2 \norm{\nab p}{L^2}^2.
\end{align*}
Integrating in $t$ over $(0,T)$ gives \reff{e2.11}.
 
Alternatively, setting $\bv=t^2\bu_{tt}$ in \reff{e2.12}, $\varphi=t^2\tp_t$ 
in \reff{e2.13} (after differentiating in $t$ one more time),
$\psi=t^2 p_t=t^2 (\tp+\alpha q)_t$ in \reff{e2.14}, and adding the 
resulted equations we have
\begin{align*}
\frac12 \frac{d}{dt}\Bigl[t^2 \beta\norm{\nab \bu_t}{L^2}^2 
+t^2 \alpha \norm{q_t}{L^2}^2 \Bigr] + t^2 \kappa \norm{\nab p_t}{L^2}^2 
= t\beta\norm{\nab \bu_t}{L^2}^2 + t\norm{q_t}{L^2}^2. 
\end{align*} 
Integrating in $t$ over $(0,T)$ and using \reff{e2.11} yield \reff{e2.11a}.

\reff{e2.11b}--\reff{e2.11d} are immediate consequences of \reff{e2.14}, 
\reff{e2.11}, and \reff{e2.11a}.

Choose $\phi_0\in C^1(\overline{\Ome})$ such that $(\phi_0,1)=C_{\tp}$.
Then $\tp-\phi_0\in L^2_0(\Ome)$. By an equivalent version of the inf-sup
condition (cf. \cite{ber}) we have
\begin{align*}
\norm{\nab(\tp-\phi_0)}{L^2} &\leq \alpha_0^{-1} \sup_{\bv\in \bH_0^1(\Ome)} 
\frac{\bigl(\nab(\tp-\phi_0), \bv\bigr)}{\norm{\nab \bv}{L^2}}  \\
&\leq \alpha_0^{-1} \sup_{\bv\in \bH_0^1(\Ome)} 
\frac{-\beta \bigl(\nab\bu, \nab\bv\bigr)}{\norm{\nab \bv}{L^2}}  
+ \alpha_0^{-1}\sqrt{d} \norm{\phi_0}{L^2} \\
&\leq  \alpha_0^{-1} \Bigl[ \beta \norm{ \nab\bu}{L^2} +\sqrt{d} \norm{\phi_0}{L^2}\Bigr].
\end{align*}
Hence,
\[
\norm{\nab\tp}{L^2} \leq \norm{\nab \phi_0}{L^2} 
+ \alpha_0^{-1} \Bigl[ \beta \norm{ \nab\bu}{L^2} +\sqrt{d} \norm{\phi_0}{L^2}\Bigr],
\]
which together with \reff{e2.10} and \reff{e2.11} imply that 
(recall $p=\tp+\alpha q$)
\begin{align*}
&\norm{\nab\tp}{L^2(L^2)} \leq 
\sqrt{T} \norm{\nab\tp}{L^\infty(L^2)}
\leq \sqrt{T} \norm{\nab \phi_0}{L^2} 
+\alpha_0^{-1}\sqrt{T}\Bigl[\beta J(0)^{\frac12}
+\sqrt{d} \norm{\phi_0}{L^2}\Bigr],\\
&\alpha \sqrt{\kappa} \norm{\nab q}{L^2(L^2)} 
\leq  J(0)^{\frac12} + \sqrt{\kappa T} \norm{\nab \phi_0}{L^2} 
+ \alpha_0^{-1}\sqrt{\kappa T}\Bigl[\beta J(0)^{\frac12}
+\sqrt{d}\norm{\phi_0}{L^2}\Bigr],\\
&\alpha \sqrt{\kappa} \norm{\sqrt{t}\nab q}{L^\infty(L^2)}
\leq  J(0)^{\frac12} + \sqrt{\kappa T} \norm{\nab \phi_0}{L^2} 
+ \alpha_0^{-1} \sqrt{\kappa T } \Bigl[ \beta J(0)^{\frac12} 
+ \sqrt{d}\norm{\phi_0}{L^2}\Bigr].
\end{align*}
Hence, \reff{e2.11e} holds.

Finally, \reff{e2.11f} follows immediately from the equation
$\beta \Del \bu= \nab \tp$ and \reff{e2.11e}. The proof is complete.
\end{proof}

\section{Semi-discrete finite element methods in space}\label{sec-3}
The goal of this section is to present the ideas and specific 
semi-discrete finite element methods for discretizing the variational
problem \reff{e2.4}--\reff{e2.7} in space based on the 
multiphysical (deformation and diffusion) approach. That is, 
we shall approximate $(\bu,\tp)$ using a (stable) Stokes 
solver and approximate $q$ by a convergent diffusion equation solver.
The combination of the Taylor-Hood mixed finite element \cite{ber} 
and the conforming $P_1$ finite element
is chosen as a specific example to present the ideas.

\subsection{Formulation of finite element methods} \label{sec-3.1}
Assume $\Ome \subset \R^d$ ($d=2,3$) is a polygonal domain.
Let  $\cT_h$ is a quasi-uniform triangulation or rectangular
partition of $\Omega \subset \R^d$ with mesh size $h$ such that
$\overline{\Omega} = \bigcup_{K \in \T} \overline{K}$.
Also, let $(\bV_h,M_h)$ be a stable mixed finite element pair, that is, 
$\bV_h \subset \bH^1(\Ome)$ and $M_h \subset L^2(\Omega)$ satisfy
the inf-sup condition
\begin{align}\label{e4.1}
\sup_{\bv_h\in \bV_h\cap \bX}\frac{(\div \bv_h,\varphi_h)}{\norm{\nab
\bv_h}{L^2}}\geq \beta_0 \norm{\varphi_h}{L^2} \qquad \forall
\varphi_h\in M_h\cap L^2_0(\Ome), \quad \beta_0>0.
\end{align}

A well-known example that satisfies \reff{e4.1} is the following
so-called Taylor-Hood element (cf. \cite{ber}):
\begin{align*}
\bV_h &=\{\bv_h\in \bC^0(\overline{\Ome});\,
\bv_h|_K\in \bP_2(K)~~\forall K\in \cT_h\}, \\
M_h &=\{\varphi_h\in C^0(\overline{\Ome});\, \varphi_h|_K\in P_1(K)~~\forall K\in
\cT_h\}.
\end{align*}
In the sequel, we shall only present the analysis for the
Taylor-Hood element, but remark that the analysis can be
readily extended to other stable combinations. On the other
hand, constant pressure space is not recommended because that
would result in no rate of convergence for the approximation 
of the pressure $p$ (see Section \ref{sec-3.3}).

Approximation space $W_h$ for $q$ variable can be chosen 
independently, any piecewise polynomial space is acceptable as long 
as $W_h \supset M_h$. Especially, $W_h\subset L^2(\Ome)$ can be 
chosen as a fully discontinuous piecewise polynomial space, 
although it is more convenient to choose $W_h$ to be a continuous 
(resp. discontinuous) space if $M_h$ is a continuous (resp. discontinuous) 
space. The most convenient and efficient choice is $W_h =M_h$, 
which will be adopted in the remaining of this paper. 

We now ready to state our semi-discrete finite element method 
for problem \reff{e2.4}--\reff{e2.7}.
Let $q_0=\div \bu_0$. We seek $(\bu_h,\tp_h,q_h):\, 
(0,T]\to \bV_h\times M_h\times W_h$ and $(\bu_h(0),q_h(0))\in 
\bV_h\times W_h$ such that for all $t\in (0,T]$ there hold
\begin{alignat}{2} \label{e4.3}
\beta\bigl( \nab \bu_h,\nab \bv_h\bigr)- \bigl( \tp_h,\div \bv_h\bigr)
&=\langle \bbf, \bv_h\rangle_{\small\rm dual} &&\qquad\forall \bv_h\in \bV_h, \\
\bigl(\div \bu_h,\varphi_h \bigr) &=\bigl(q_h,\varphi_h \bigr)
&&\qquad\forall \varphi_h \in M_h,  \label{e4.4}\\
\bigl(\nab\bu_h(0), \nab\bw_h\bigr) = \bigl(\nab\bu_0, \nab\bw_h\bigr),
\quad \langle \bu_h(0), \nu\rangle &=\langle \bu_0, \nu\rangle 
&&\qquad\forall \bw_h\in \bV_h, \label{e4.5a}\\
\bigl(q_{ht},\psi_h \bigr)_{\small\rm dual} + \kappa \bigl(\nab (\tp_h+\a q_h),
\nab\psi_h\bigr) &=0 &&\qquad\forall \psi_h \in W_h, \label{e4.2}\\
\bigl(q_h(0), \chi_h\bigr) &= \bigl(q_0, \chi_h\bigr) 
&&\qquad\forall \chi_h\in W_h. \label{e4.2b}
\end{alignat}
Where $q_{ht}$ denotes the time derivative of $q_h$.

\begin{remark}
(a) We note that \reff{e4.5a} and \reff{e4.2b} defines 
$\bu_h(0)=\cR^h \bu_0$ and $q_h(0)=\cQ^h q_0$ (see their
definitions below).
It is easy to see that in general $q_h(0)\neq \div \bu_h(0)$
although $q(0)=\div \bu(0)$. But it is not hard
to enforce $q_h(0)=\div \bu_h(0)$ by defining $\bu_h(0)$ slightly 
differently (in the case of continuous $M_h$) or simply substituting 
\reff{e4.2b} by $q_h(0):=\div \bu_h(0)$ (in the case of discontinuous
$M_h$), even such a modification is not necessary for the sake of convergence
(see Section \ref{sec-3.3}). We also note that $\bu_h(0)=\cR^h \bu_0$
can be replaced by the $L^2$ projection $\bu_h(0)=\cQ^h \bu_0$,
the only ``drawback" of using the $L^2$ projection is that it produces
a larger error constant.

(b) In the case that $M_h$ and/or $W_h$ are discontinuous spaces,
since $M_h\not\subset H^1(\Ome)$ and/or $W_h\not\subset H^1(\Ome)$,
then the second term on the left-hand side of \reff{e4.2} must 
be modified into a sum over all elements of the integrals defined 
on each element $K\in \cT_h$, and additional jump terms on element
edges may also need to be introduced to ensure convergence.
\end{remark}

We conclude this subsection by citing a few well-known 
facts about the finite element functions. First, 
we recall the following inverse inequality for polynomial functions
\cite{cia}:
\begin{align}\label{e4.0}
\|\nab \phi_h\|_{L^2(K)}\le c_0 h^{-1}\|\phi_h\|_{L^2(K)} 
\qquad\forall \phi_h\in P_r(K),\, K\in \cT_h. 
\end{align}

Second, for any $\phi\in L^2(\Ome)$, we define its $L^2$ 
projection $\cQ^h \phi\in W_h$ as follows:
\begin{align*} 
(\cQ^h \phi, \psi_h) &=(\phi,\psi_h) \qquad \psi_h\in W_h.
\end{align*}
It is well-known that the projection operator $\cQ^h:L^2(\Ome)\to W_h$ 
satisfies (cf. \cite{bs08}), for any $\phi\in H^s(\Omega)\, (s\geq 1)$
\begin{align}\label{e4.00}
\|\cQ^h \phi-\phi\|_{L^2}+h\|\nabla(\cQ^h \phi-\phi)\|_{L^2}\le
Ch^\ell\|\phi \|_{H^\ell}, \quad \ell=\min \{2, s\}.
\end{align}
We remark that in the case $W_h\not\subset H^1(\Ome)$, 
the second term on the left-hand side of the above 
inequality has to be replaced by the broken $H^1$-norm.

Finally, for any $\bv\in \bH^1(\Ome)$ we define its elliptic projection 
$\cR^h \bv\in \bV_h$ by
\begin{align*} 
\bigl(\nab\cR^h \bv, \nab\bw_h) &=(\nab\bv, \nab\bw_h) 
\qquad \bw_h\in \bV_h, \\
\langle \cR^h \bv, 1 \rangle &=\langle \bv, 1 \rangle. 
\end{align*}
Also, for any $\phi\in H^1(\Ome)$, we define its elliptic projection
$\cS^h \phi\in W_h$ by
\begin{align*}
\bigl(\nab\cS^h \phi, \nab\psi_h) &=(\nab\phi, \nab\psi_h) 
\qquad \psi_h\in W_h, \\
(\cS^h \phi, 1 ) &=( \phi, 1 ). 
\end{align*}
It is well-known that the projection operators $\cR^h:\bH^1(\Ome)\to 
\bV_h$ and $\cS^h: H^1(\Ome)\to W_h$ satisfy (cf. \cite{bs08,cia}), for any 
$\bv\in H^s(\Omega), \phi\in H^s(\Ome)\, (s\geq 1)$, $m=\min \{3, s\}$
and $\ell=\min \{2, s\}$
\begin{align}\label{e4.01}
h^{-1} \|\cR^h \bv-\bv\|_{H^{-1}}+ \|\cR^h \bv-\bv\|_{L^2}
+h\|\nabla(\cR^h \bv-\bv)\|_{L^2} \le Ch^m\|\bv \|_{H^m}.\\
h^{-1} \|\cS^h \phi-\phi\|_{H^{-1}}+ \|\cS^h \phi-\phi\|_{L^2}
+h\|\nabla(\cS^h \phi-\phi)\|_{L^2} \le Ch^\ell\|\phi \|_{H^\ell}.\label{e4.02}
\end{align}

\subsection{Stability and solvability of \reff{e4.3}--\reff{e4.2b}}\label{sec-3.2}
In this subsection, we shall prove that the semi-discrete 
solution $(\bu_h,\tp_h,q_h)$ defined in the previous subsection satisfies 
an energy law similar to \reff{e2.10}. In addition, they satisfy the 
same conservation laws as those enjoyed by their
continuous counterparts $(\bu,\tp,q)$ (see Lemma \ref{lem2.2}).
An immediate consequence of the stability and the conservation laws
is the well-posedness of problem \reff{e4.3}--\reff{e4.2b}.
Moreover, the stability also serves as a step stone for us 
to establish convergence results in the next subsection.

\begin{lemma}\label{lem3.1a}
Let $(\bu_h,\tp_h,q_h)$ be a solution of \reff{e4.3}--\reff{e4.2b}
and set $p_h:=\tp_h+\alpha q_h$. Then there holds the following energy law:
\begin{align}\label{eq3.12}
J_h(t)+\kappa \int_0^t \norm{\nab p_h(s)}{L^2}^2\, ds
=J_h(0) \qquad\forall t\in (0,T],
\end{align}
where
\[
J_h(t):= \frac12 \Bigl[ \beta \norm{\nab \bu_h(t)}{L^2}^2 +\alpha \norm{q_h(t)}{L^2}^2
      - 2\langle \bbf, \bu_h(t) \rangle_{\small\rm dual} \Bigr].
\]
\end{lemma}

\begin{proof}
If $\bu_{ht}\in L^2((0,T);\bH^1(\Ome))$, then \reff{eq3.12} follows immediately
from setting $\bv_h=\bu_{ht}$ in \reff{e4.3},  $\varphi_h=\tp_h$ after
differentiating \reff{e4.4} with respect to $t$,  $\psi_h=p_h$,
and adding the resulted equations. On the other hand, 
if $\bu_{ht}\not\in L^2((0,T);H^1(\Ome))$, then $\bv_h=\bu_{ht}$
is not a valid test function. This technical difficulty can be
easily overcome by smoothing $\bu_h$ in $t$ through 
a symmetric mollifier as described in the proof of Lemma \ref{lem2.1}.
\end{proof}

\begin{lemma}\label{lem3.2a}
Every solution $(\bu_h,\tp_h,q_h)$ of \eqref{e4.3}--\eqref{e4.2b} 
satisfies the following conservation laws:
\begin{align}\label{eq3.13}
\bigl(q_h, 1\bigr) =C_q,  \quad
\langle \bu_h, \nu\rangle = C_{\bu},  \quad
\bigl(\tp_h, 1\bigr) =C_{\tp}, \quad 
\bigl(p_h, 1\bigr) =C_p.
\end{align}
\end{lemma}

\begin{proof}
\reff{eq3.13}$_1$ follows from setting $\psi_h=1$ in
\reff{e4.2} and $\chi_h=1$ is \reff{e4.2b}.  \reff{eq3.13}$_2$ 
follows from taking $\varphi_h= 1$ in \reff{e4.4}. 
\reff{eq3.13}$_3$ derives from letting $\bv_h=x$ in 
\reff{e4.3}. Finally, \reff{eq3.13}$_4$ follows
from \reff{eq3.13}$_1$, \reff{eq3.13}$_3$, and 
the definition of $p_h$.
\end{proof}

\medskip
With the help of the above two lemmas, we can prove the
solvability of problem \eqref{e4.3}--\eqref{e4.2b}.

\begin{theorem}\label{thm3.1a}
Suppose $(\bu_0, \bbf) \in \bH^1(\Ome) \times \bH^{-\frac12}(\p\Omega)$,
and $\langle \bbf, 1 \rangle_{\small\rm dual} =0$.
then \eqref{e4.3}--\eqref{e4.2b} has a unique solution.
\end{theorem}

\begin{proof}
Since \eqref{e4.3}--\eqref{e4.2b} can be written as an initial value
problem for a system of linear ODEs, the existence of solutions
follows immediately from the linear ODE theory. 

To show the uniqueness, it suffices to prove that 
the problem only has the trivial solution if $\bbf=\bu_0=0$.  
For the zero sources, the energy law and the $H^1$-norm stability 
of the elliptic projection,
\[
\norm{\nab\bu_h(0)}{L^2}=\norm{\nab\cR^h \bu(0)}{L^2} 
\leq \norm{\nab\bu(0)}{L^2},
\]
immediately imply that $\bu,\tp$, and $q$ are constant functions. Since
$C_{\bu}=C_q=C_{\tp}=0$ as $\bbf=\bu_0=0$, hence, 
$\bu,\tp$, and $q$ must identically equal zero. The proof is complete. 
\end{proof}

By exploiting the linearity of equations \reff{e4.3}--\reff{e4.2b},
we can prove certain energy estimates for the time derivatives \
of the solution $(\bu_h, \tp_h, q_h)$. 

\begin{theorem}\label{regularity_1} 
The solution $(\bu_h, \tp_h, q_h)$ of the semi-discrete method
and $p_h:=\tp_h+\alpha q_h$ satisfy the following estimates:
\begin{align}\label{eq3.14}
\norm{(q_h)_t}{L^2(H^{-1})} &\leq J_h(0)^{\frac12}. \\
\sqrt{\beta} \norm{\sqrt{t} (\nab \bu_h)_t}{L^2(L^2)}
+\sqrt{\alpha} \norm{\sqrt{t} (q_h)_t}{L^2(L^2)}
+\sqrt{\kappa} \norm{\sqrt{t} \nab p_h}{L^\infty(L^2)} 
&\leq J_h(0)^{\frac12}. \label{eq3.15}  \\
\sqrt{\beta} \norm{t (\nab \bu_h)_t}{L^\infty(L^2)}
+\sqrt{\alpha} \norm{t (q_h)_t}{L^\infty(L^2)}
+\sqrt{\kappa} \norm{t (\nab p_h)_t}{L^2(L^2)} 
&\leq [2J_h(0)]^{\frac12}. \label{eq3.16} \\
\norm{t(q_h)_{tt}}{L^2(H^{-1})} &\leq [2J_h(0)]^{\frac12}.  \label{eq3.17} 
\end{align}

\end{theorem}

Because the proofs of \reff{eq3.14}--\reff{eq3.17} follow exactly the
same lines of the proofs of their differential counterparts 
given in Theorem \ref{regularity}, we omit them.

\subsection{Convergence analysis}\label{sec-3.3}

Define the error functions
\[
\bE_{\bu}(t):=\bu(t)-\bu_h(t),\quad
E_{\tp}(t):=\tp(t)-\tp_h(t),\quad
E_q(t):=q(t)-q_h(t),
\]
and 
\[
p(t):=\tp(t)+\alpha q(t),\quad
p_h(t):=\tp_h(t)+\alpha q_h(t),\quad
E_p(t):= p(t)-p_h(t).
\]
Trivially, $ E_p(t)=E_{\tp}(t)+\alpha E_q(t)$.  

Subtracting each of \reff{e4.3}--\reff{e4.2b} from their
respective counterparts in \reff{e2.4}--\reff{e2.7} we obtain 
the following error equations:
\begin{alignat}{2} \label{e4.10}
\beta\bigl( \nab \bE_{\bu},\nab \bv_h\bigr)- \bigl( E_{\tp},\div \bv_h\bigr)
&=0 &&\qquad\forall \bv_h\in \bV_h, \\
\bigl(\div \bE_{\bu},\varphi_h \bigr) &=\bigl(E_q,\varphi_h \bigr)
&&\qquad\forall \varphi_h \in M_h,  \label{e4.11}\\
\bu_h(0) &=\cR^h \bu(0), &&\qquad \label{e4.11a} \\
\bigl((E_q)_t,\psi_h \bigr)_{\small\rm dual} 
+\kappa \bigl(\nab (E_{\tp} +\a E_q), \nab\psi_h\bigr)
&=0 &&\qquad\forall \psi_h \in W_h, \label{e4.12}\\
q_h(0) &= \cQ^h q(0). &&\qquad \label{e4.12b}
\end{alignat}

Introduce the following decomposition of error functions
\begin{alignat*}{3}
&\bE_{\bu}=\mathbf{\Lam}_{\bu} + \mathbf{\Theta}_{\bu}, 
&&\quad \bLamu:=\bu-\cR^h\bu, &&\quad \bThetau:=\cR^h\bu -\bu_h,\\
&E_{\tp}=\Lam_{\tp} + \Theta_{\tp}, 
&&\quad \Lam_{\tp}:=\tp-\cQ^h\tp, &&\quad \Theta_{\tp}:=\cQ^h\tp -\tp_h,\\
&E_q=\Lam_q + \Theta_q, 
&&\quad \Lam_q:=q-\cQ^h q, &&\quad \Theta_q:=\cQ^h q -q_h,\\
&E_p=\Lam_p + \Theta_p, 
&&\quad \Lam_p:=p-\cQ^h p, &&\quad \Theta_p:=\cQ^h p -p_h,\\
&E_p=\Psi_p + \Phi_p, 
&&\quad \Psi_p:=p-\cS^h p, &&\quad \Phi_p:=\cS^h p -p_h,\\
\end{alignat*}
Trivially, $\Phi_p= \Lam_p-\Psi_p + \Theta_p, 
\Theta_p=\Theta_{\tp}+\alpha\Theta_q$.  

By the definition of $\cR^h$ and $\cQ^h$, the above error equations can be
written as
\begin{alignat}{2} \label{eq3.19}
\beta\bigl( \nab \bThetau,\nab \bv_h\bigr)- \bigl(\Theta_{\tp},\div \bv_h\bigr)
&=\bigl(\Lam_{\tp},\div \bv_h\bigr) &&\qquad\forall \bv_h\in \bV_h, \\
\bigl(\div \bThetau,\varphi_h \bigr) &=\bigl(\Theta_q,\varphi_h \bigr) 
- \bigl(\div \bLamu,\varphi_h \bigr) 
&&\qquad\forall \varphi_h \in M_h,  \label{eq3.20}\\
\bThetau(0) &= 0 &&\qquad \label{eq3.21}\\
\bigl((\Theta_q)_t,\psi_h \bigr)_{\small\rm dual} 
+\kappa \bigl(\nab \Phi_p, \nab\psi_h\bigr)
&=0 &&\qquad\forall \psi_h \in W_h, \label{eq3.22}\\
\Theta_q(0) &=0. &&\qquad \label{eq3.23}
\end{alignat}
So $\bThetau, \Theta_p, \Theta_{\tp}$ and $\Theta_q$ satisfy 
the same type of equations as $\bu_h,\tp_h, p_h$ and $q_h$ do,
except the terms containing $\bLamu$ and $\Lam_{\tp}$ 
on the right-hand sides of the equations. Hence, 
$\bThetau, \Theta_p, \Theta_{\tp}$ and $\Theta_q$ are expected
to satisfy an equality similar to the energy law \reff{eq3.12}.

To the end, taking $\bv_h=(\bThetau)_t,\varphi_h=\Theta_{\tp}$ 
(after differentiating \reff{eq3.20} with respect to $t$),
$\psi_h=\Phi_p=  \Lam_p-\Psi_p +\Theta_{\tp}+\alpha\Theta_q$
in \reff{eq3.19}--\reff{eq3.23}, and adding 
the resulted equations we obtain 
\begin{align}\label{e4.13}
&\frac12 \frac{d}{dt}\bigl[ \beta \|\nab \bThetau\|_{L^2}^2
+\alpha \|\Theta_q\|_{L^2}^2\bigr]
+\kappa \|\nab \Phi_p\|_{L^2}^2   \\
&\hskip 0.5in
=\bigl(\Lam_{\tp}, \div (\bThetau)_t \bigr) 
-\bigl(\div (\bLamu)_t, \Theta_{\tp} \bigr)
+\bigl((\Theta_q)_t, \Psi_p-\Lam_p \bigr) \no  \\
&\hskip 0.5in
= \frac{d}{dt} \Big[ \bigl(\Lam_{\tp}, \div \bThetau \bigr) 
+\bigl(\Theta_q, \Psi_p-\Lam_p \bigr) \Bigr]
- \bigl((\Lam_{\tp})_t, \div \bThetau \bigr)  \no \\
&\hskip 1.9in
-\bigl(\Theta_q, (\Psi_p-\Lam_p)_t \bigr)
-\bigl(\div(\bLamu)_t, \Theta_{\tp} \bigr). \no
\end{align}

Integrating in $t$ and using Schwarz inequality we get
\begin{align*}
&\frac12 \bigl[ \beta \|\nab \bThetau(t)\|_{L^2}^2
+\alpha \|\Theta_q(t)\|_{L^2}^2\bigr]
+\kappa \int_0^t \|\nab \Phi_p(s)\|_{L^2}^2\, ds   \\
&\quad 
= 
\bigl(\Lam_{\tp}(t), \div \bThetau(t) \bigr) 
+\bigl(\Theta_q(t), \Psi_p(t)-\Lam_p(t) \bigr)
-\int_0^t \bigl(\Theta_q, (\Psi_p(s)-\Lam_p(s))_t \bigr)\, ds \\
&\qquad
-\int_0^t \Bigl[ \bigl((\Lam_{\tp}(s))_s, \div \bThetau(s) \bigr) 
+\bigl(\div(\bLamu(s))_s, \Theta_{\tp}(t) \bigr) \Bigr]\, ds\\
&\quad
\leq 
\frac{\beta}4 \norm{\nab \bThetau(t)}{L^2}^2 
+\frac{1}{\beta} \norm{\Lam_{\tp}(t)}{L^2}^2 
+\frac{\alpha}4 \|\Theta_q(t)\|_{L^2}^2 \\
&\qquad
+\frac{1}{\alpha} \bigl[\norm{\Psi_p(t)}{L^2}^2+\norm{\Lam_p(t)}{L^2}^2\bigr] 
+\frac{\epsilon_0 \kappa}2 \int_0^t \|\Phi_p(s)\|_{L^2}^2\, ds \\
&\qquad
+\int_0^t \bigl[ \beta \norm{\nab\bThetau(s)}{L^2}^2 
+\alpha \norm{\Theta_q(s)}{L^2}^2 \bigr]\,ds
+\frac{1}{\alpha} \int_0^t \bigl[\norm{(\Psi_p(s)
-\Lam_p(s))_s}{L^2}^2\bigr]\, ds \\
&\qquad
+ C\int_0^t \Bigl[ \frac{1}{\beta} \norm{(\Lam_{\tp}(s))_s}{L^2}^2 
+\bigl(\alpha+\frac{1}{\epsilon_0\kappa}\bigr) \norm{\div(\bLamu(s))_s}{L^2}^2
\Bigr]\,ds.
\end{align*}
Where we have used the facts that $\bThetau(0)=0$, $\Theta_q(0)=0$, and 
$\Theta_{\tp}=\Theta_p-\alpha\Theta_q =\Psi_p-\Lam_p+\Phi_p-\alpha\Theta_q$
to get the second inequality. $\epsilon_0$ is a positive
constant to be chosen later.
Since $(\Phi_p,1)=0$, using Poinca\'e's inequality 
\[
\norm{\Phi_p}{L^2} \leq c_1 \norm{ \nab \Phi_p}{L^2},
\]
and choosing $\epsilon=c_1^{-1}$ the above inequality can be written as
\begin{align}\label{eq3.24}
\beta \|\nab \bThetau(t)\|_{L^2}^2 &+\alpha \|\Theta_q(t)\|_{L^2}^2
+\kappa \int_0^t \|\nab \Phi_p(s)\|_{L^2}^2\, ds  \\
&\leq \int_0^t \bigl[ \beta \norm{\nab\bThetau(t)}{L^2}^2 
+\alpha \norm{\Theta_q(t)}{L^2}^2 \bigr]\,ds + \mathcal{A}(t;\bu,\tp,p,q), \no
\end{align}
where
\begin{align*}
\mathcal{A}(t;\bu,\tp,p,q) &:= 
\frac{4}{\beta} \norm{\Lam_{\tp}(t)}{L^2}^2 
+\frac{1}{\alpha} \bigl[\norm{\Psi_p(t)}{L^2}^2+\norm{\Lam_p(t)}{L^2}^2\bigr]\\
&\qquad
+C\int_0^t \Bigl[\frac{1}{\beta} \norm{(\Lam_{\tp}(s))_s}{L^2}^2 
+\bigl(\alpha+\frac{c_1}{\kappa}\bigr) \norm{\div(\bLamu(s))_s}{L^2}^2 \Bigr]\,ds\\
&\qquad
+\frac{4}{\alpha} \int_0^t \bigl[\norm{(\Psi_p(s))_s}{L^2}^2 
+\norm{(\Lam_p(s))_s}{L^2}^2 \bigr] \, ds. 
\end{align*}

From \reff{e4.00}--\reff{e4.02} we deduce
\begin{align}\label{eq3.26}
\mathcal{A}(t;\bu,\tp,p,q) \leq C_1(T;\bu,\tp,p,q)\, h^2, \\
\mathcal{A}(t;\bu,\tp,p,q) \leq C_2(T;\bu,\tp,p,q)\, h^4,\label{eq3.26a}
\end{align}
where
\begin{align*}
C_1(T;\bu,\tp,p,q) &:=C\bigl[ 
\beta^{-1} \norm{\tp}{L^\infty(H^1)}^2 + \beta^{-1} \norm{\tp_t}{L^2(H^1)}^2 \\
&\qquad +(\alpha + c_1\kappa^{-1}) \norm{(\div\bu)_t}{L^2(H^1)}^2
+ \alpha^{-1} \bigl[ \norm{p}{L^\infty(H^1)}^2 + \norm{p_t}{L^2(H^1)}^2 \bigr],\\
C_2(T;\bu,\tp,p,q) &:=C\bigl[ 
\beta^{-1} \norm{\tp}{L^\infty(H^2)}^2 + \beta^{-1} \norm{\tp_t}{L^2(H^2)}^2 \\
&\qquad +(\alpha +c_1 \kappa^{-1}) \norm{(\div\bu)_t}{L^2(H^2)}^2
+ \alpha^{-1} \bigl[ \norm{p}{L^\infty(H^2)}^2 + \norm{p_t}{L^2(H^2)}^2 \bigr].
\end{align*}

It follows from an application of the Gronwall's lemma to \reff{eq3.24} that
\begin{align}\label{eq3.27}
\max_{0\leq t\leq T} \bigl[ \beta \|\nab \bThetau(t)\|_{L^2}^2 
+\alpha \|\Theta_q(t)\|_{L^2}^2 \bigr]
+\kappa \int_0^T \|\nab \Phi_p(s)\|_{L^2}^2\, ds  
\leq \mathcal{A}(T;\bu,\tp,q)\,e^T.
\end{align}

An application of the triangle inequality yields the following 
main theorem of this section.

\begin{theorem}\label{theorem3.1}
Assume $\bu_0\in \bH^1(\Ome)$. Let
\begin{align*}
\widehat{C}_1(T;\bu,p,q) &:= C\bigl[ \beta \norm{\bu}{L^\infty(H^2)}^2 
   + \alpha \norm{q}{L^\infty(H^1)}^2 + \kappa  \norm{p}{L^2(H^2)}^2 \bigr],\\
\widehat{C}_2(T;\bu,q) &:= C\bigl[ \beta \norm{\bu}{L^\infty(H^3)}^2 
   + \alpha \norm{q}{L^\infty(H^2)}^2 \bigr].
\end{align*}
Suppose that $C_1(T;\bu,\tp,p,q) < \infty$ and $\widehat{C}_1(T;\bu,p,q)<\infty$, 
then there holds error estimate
\begin{align} \label{e4.18}
&\max_{0\leq t\leq T} \bigl[ \sqrt{\beta}\|\nab(\bu(t)-\bu_h(t))\|_{L^2}
+\sqrt{\alpha}\|q(t)-q_h(t)\|_{L^2} \bigr] \\
&\qquad
+\Bigl( \kappa \int^T_0\|\nab(p-p_h)(t)\|_{L^2}^2\,dt\Bigr)^{\frac12} 
\leq \bigl[ C_1(T;\bu,\tp,p,q)^{\frac12}\,e^{\frac{T}{2}}
+\widehat{C}_1(T;\bu,p,q)^{\frac12} \bigr] h, \no
\end{align}
Moreover, if $C_2(T;\bu,\tp,p,q) < \infty$ and $\widehat{C}_2(T;\bu,q)<\infty$,
then there also holds
\begin{align} \label{e4.19}
&\max_{0\leq t\leq T} \bigl[ \sqrt{\beta}\|\nab(\bu(t)-\bu_h(t))\|_{L^2}
+\sqrt{\alpha}\|q(t)-q_h(t)\|_{L^2} \bigr] \\
&\hskip 2.2in
\leq \bigl[ C_2(T;\bu,\tp,p,q)^{\frac12}\,e^{\frac{T}{2}}
+\widehat{C}_2(T;\bu,q)^{\frac12} \bigr] h^2, \no
\end{align}

\end{theorem}

\begin{remark}
(a) We note that the above error estimates are 
optimal, and $\norm{\nab \Phi_q}{L^2(L^2)}=\norm{\nab(p-\cS^h p)}{L^2(L^2)}$ 
enjoys a superconvergence when the PDE solution is regular.
It is also interesting to remark that the initial errors at $t=0$ 
do not appear in the above error bounds.

(b) In light of Theorem \ref{regularity}, the regularity 
assumptions of Theorem \ref{theorem3.1} are valid if the domain
$\Ome$ and datum functions $\bbf$ and $\bu_0$ are sufficient regular.
\end{remark}

\section{Fully discrete finite element methods}\label{sec-4} 
\subsection{Formulation of fully  discrete finite element methods}\label{sec4.1}
In this section, we consider space-time discretization which combines
the time-stepping scheme of Algorithm 1 with the (multiphysical) 
spatial discretization developed in the previous section. 
We first prove that under the mesh constraint $\Del t =O(h^2)$ 
the fully discrete solution satisfies a discrete 
energy law which mimics the differential energy laws
\reff{e2.10} and \reff{e2.9}. We then derive optimal order error estimates
in various norms for the numerical solution, which, as expected,
are of the first order in $\Del t$.

Using the time-stepping scheme of Algorithm $1$
to discretize \reff{e4.3}--\reff{e4.2b} we get the
following fully discrete finite element method for problem 
\reff{e2.4}--\reff{e2.7}.

\medskip
{\bf Fully discrete version of Algorithm 1:}
\medskip

\begin{itemize}
\item[(i)]  Compute $q_h^0\in W_h$ and $\bu_h^0\in \bV_h$ by
\begin{align} \label{eq4.1}
\bigl(q_h^0, \chi_h\bigr)&=\bigl(q_0, \chi_h\bigr) 
&&\quad\forall \chi_h\in W_h. \\
\bigl(\nab\bu_h^0, \nab\bw_h\bigr) = \bigl(\nab\bu_0, \nab\bw_h\bigr),
\quad \langle \bu_h^0, \nu\rangle &=\langle \bu_0, \nu\rangle 
&&\quad\forall \bw_h\in \bV_h, \label{eq4.2}
\end{align}
\item[(ii)] For $n=0,1,2, \cdots$, do the following two steps

{\em Step 1:} Solve for $(\bu_h^{n+1},\tp_h^{n+1})\in \bV_h\times M_h$
such that
\begin{alignat}{2} \label{eq4.3}
\beta \bigl( \nab \bu_h^{n+1},\nab \bv_h\bigr)
-\bigl(\tp_h^{n+1},\div \bv_h\bigr)
&=\langle \bbf, \bv_h\rangle_{\small\rm dual}
&&\qquad\forall \bv_h\in \bV_h, \\
\bigl(\div \bu_h^{n+1},\varphi_h \bigr) &=\bigl(q_h^n,\varphi_h \bigr)
&&\qquad\forall \varphi_h \in M_h,  \label{eq4.4} \\
\bigl(\tp_h^{n+1}, 1\bigr) =C_{\tp}, \quad
\langle \bu_h^{n+1}, \nu\rangle &= C_{\bu}.  &&  \label{eq4.5}
\end{alignat}

{\em Step 2:} Solve for $q_h^{n+1}\in W_h$ such that
\begin{alignat}{2} \label{eq4.6}
\bigl(d_t q_h^{n+1},\psi_h \bigr)
+\kappa \bigl(\nab (\alpha q_h^{n+1}+\tp_h^{n+1}), \nab\psi_h \bigr) &=0
&&\qquad\forall \psi_h \in W_h,\\
\bigl(q_h^{n+1}, 1\bigr) & =C_{q}. &&  \label{eq4.7}
\end{alignat}
\end{itemize}

Some remarks need to be given before analyzing the algorithm.

\begin{remark} \label{rem3.1}
(a) At each time step, problem \reff{eq4.3}--\reff{eq4.5} in 
{\em Step 1} of (ii) solves a generalized Stokes problem 
with a slip boundary condition for $\bu$. Two conditions 
in \reff{eq4.5} are imposed to ensure the uniqueness of
the solution. Well-posedness of the Stokes problem follows
easily with help of the {\em inf-sup} condition \reff{e4.1}.

(b) $q_h^{n+1}$ is clearly well-defined in {\em Step 2} of (ii).

(c) The algorithm does not produce $\tp_h^0$ (and $p_h^0$), which is not
needed to execute the algorithm.

(d) Reversing the order of {\em Step 1} and {\em Step 2} in 
(ii) of the above algorithm yields the following alternative 
algorithm.

\medskip
{\bf Fully Discrete Finite Element Algorithm 2:}
\medskip

\begin{itemize}
\item[(i)] Compute $q_h^0\in W_h$ and $(\bu_h^0,\tp_h^0)\in \bV_h\times M_h$ by
\begin{alignat*}{2}
\bigl(q_h^0, \chi_h\bigr) &=\bigl(q_0, \chi_h\bigr) 
&&\qquad\forall \chi_h\in W_h. \\
\beta \bigl( \nab \bu_h^0,\nab \bv_h\bigr)
-\bigl(\tp_h^0,\div \bv_h\bigr)
&=\langle \bbf, \bv_h\rangle_{\small\rm dual}
&&\qquad\forall \bv_h\in \bV_h, \\
\bigl(\div \bu_h^0,\varphi_h \bigr) &=\bigl(q_h^0,\varphi_h \bigr)
&&\qquad\forall \varphi_h \in M_h,  \\
\bigl(\tp_h^0, 1\bigr) =C_{\tp}, \quad
\langle \bu_h^0, \nu\rangle &= C_{\bu}.  && 
\end{alignat*}

\item[(ii)] For $n=0,1,2, \cdots$, do the following two steps

{\em Step 1:} Solve for $q_h^{n+1}\in W_h$ such that
\begin{alignat*}{2} 
\bigl(d_t q_h^{n+1},\psi_h \bigr)
+\kappa \bigl(\nab (\alpha q_h^{n+1}+\tp_h^n), \nab\psi_h \bigr) &=0
&&\qquad\forall \psi_h \in W_h,\\
\bigl(q_h^{n+1}, 1\bigr) & =C_{q}. && 
\end{alignat*}

{\em Step 2:} Solve for $(\bu_h^{n+1},\tp_h^{n+1})\in \bV_h\times M_h$
such that
\begin{alignat*}{2}
\beta \bigl( \nab \bu_h^{n+1},\nab \bv_h\bigr)
-\bigl(\tp_h^{n+1},\div \bv_h\bigr)
&=\langle \bbf, \bv_h\rangle_{\small\rm dual}
&&\qquad\forall \bv_h\in \bV_h, \\
\bigl(\div \bu_h^{n+1},\varphi_h \bigr) &=\bigl(q_h^{n+1},\varphi_h \bigr)
&&\qquad\forall \varphi_h \in M_h,  \\
\bigl(\tp_h^{n+1}, 1\bigr) =C_{\tp}, \quad
\langle \bu_h^{n+1}, \nu\rangle &= C_{\bu}.  && 
\end{alignat*}
\end{itemize}
We note that Algorithm 2 requires the starting values $q_h^0$ and
$(\bu_h^0,\tp_h^0)$, and the latter is generated by solving a
generalized Stokes problem at $t=0$. In general, $\bu_h^0$ of
Algorithms 1 and 2 are different. Later we will remark that 
Algorithm 2 also provides a convergent scheme.
\end{remark}

\subsection{Stability analysis of fully discrete Algorithm 1}\label{sec-4.2}
The primary goal of this subsection is to derive a discrete energy law
which mimics the PDE energy law \reff{e2.9} and the semi-discrete 
energy law \reff{eq3.12}. It turns out that such a discrete energy law
only holds if $h$ and $\Del t$ satisfy the mesh constraint
$\Del t=O(h^2)$. This mesh constraint is the cost of using
the time delay decoupling strategy in \reff{eq4.4}.

Before discussing the stability of the fully discrete Algorithm 1,
We first show that the constraints \reff{eq4.5} an \reff{eq4.7}
are consistent with the equations \reff{eq4.3}, \reff{eq4.4},
and \reff{eq4.6}. 

\begin{lemma}\label{lem3.1}
Let $\{(\bu_h^n,\tp_h^n,q_h^n)\}_{n\geq 1}$ satisfy 
\reff{eq4.3}, \reff{eq4.4},
and \reff{eq4.6}. Define $p_h^{n}:=\tp_h^n+\alpha q_h^{n-1}$ 
for $n=1,2,3,\cdots$.  Then there hold
\begin{alignat}{2}\label{e3.5}
\bigl(q_h^n,1\bigr) &= C_q  &&\qquad\mbox{for } n=0,1,2,\cdots,\\
\langle \bu_h^n,\nu \rangle &= C_{\bu}=C_q &&\qquad\mbox{for } n=1,2,3,\cdots,
\label{e3.6} \\
\bigl(\tp_h^n,1\bigr) &= C_{\tp} &&\qquad\mbox{for } n=1,2,3,\cdots,\label{e3.7}\\
\bigl(p_h^n,1\bigr) &= C_{p} &&\qquad\mbox{for } n=1,2,3,\cdots,\label{e3.8}
\end{alignat}
where $C_q, C_{\bu}, C_{\tp}$, and $C_p$ are defined by
\reff{e1.14}--\reff{e1.16} and \reff{e1.18}
\end{lemma}

\begin{proof}
Taking $\psi_h\equiv 1$ in \reff{eq4.6} yields
\[
\bigl(d_t q_h^{n+1},1 \bigr)=0.
\]
Hence,
\[
\bigl( q_h^{n+1},1 \bigr)=\bigl( q_h^0,1 \bigr)
=\bigl( q_0,1 \bigr)= C_q \qquad\mbox{for } n=0,1,2,\cdots.
\]
So \reff{e3.5} holds. \reff{e3.6} immediately follows from 
setting $\varphi_h\equiv 1$ in \reff{eq4.4} and using \reff{e3.5}.

As in the time-continuous case, \reff{e3.7} follows from taking
$\bv_h=x$ in \reff{eq4.3} and using the fact that $\div x=d$ and $\nab x=I$.

Finally, \reff{e3.8} is an immediate consequence of the definition
$p_h^{n}:=\tp_h^n+\alpha q_h^{n-1}$, \reff{e3.5}, and \reff{e3.7}. The proof is
complete.
\end{proof}


Next lemma establishes an identity which mimics the 
continuous energy law for the fully discrete solution of Algorithm 1. 

\begin{lemma}\label{lemma3.1}
Let $\{(\bu_h^n,\tp_h^n,q_h^n)\}_{n\geq 1}$ be a solution of
\reff{eq4.1}--\reff{eq4.7}. Define $p_h^n:=\tp_h^n+\alpha q_h^{n-1}$ for $n\geq 1$. 
Then there holds the following identity:
\begin{align}\label{e3.13}
&\mathcal{J}_h^{\ell+1} 
+ \Del t\sum_{n=0}^\ell \Bigl[ \kappa \norm{\nab p_h^{n+1}}{L^2}^2
+\frac{\Del t}{2} \bigl( \beta \norm{d_t \nab\bu_h^{n+1}}{L^2}^2
+\alpha \norm{d_t q_h^{n}}{L^2}^2 \bigr) \Bigr] \\
&\hskip 2in
=\mathcal{J}_h^0-\kappa (\Del t)^2\sum_{n=0}^\ell \bigl( d_t\nab \tp_h^{n+1},
\nab p_h^{n+1} \bigr), \nonumber
\nonumber
\end{align}
where
\begin{align}\label{e3.14}
\mathcal{J}_h^{\ell}:= \frac12 \Bigl[ \beta \norm{\nab \bu_h^{\ell+1}}{L^2}^2
+ \alpha \norm{q_h^{\ell}}{L^2}^2 
- 2\langle \bbf, \bu_h^{\ell+1}\rangle_{\small\rm dual} \Bigr],
 \qquad\ell=0,1,2,\cdots
\end{align}

\end{lemma}

\begin{proof}
Setting $\bv_h=d_t \bu_h^{n+1}$ in \reff{eq4.3} gives
\begin{equation}\label{e3.9}
\frac{\beta}{2} d_t \norm{\nab \bu_h^{n+1}}{L^2}^2
+\frac{\beta \Del t}{2} \norm{d_t \nab\bu_h^{n+1}}{L^2}^2
= d_t \langle \bbf, \bu_h^{n+1}\rangle_{\small\rm dual}
 +\bigl(\tp_h^{n+1},\div d_t\bu_h^{n+1}\bigr).
\end{equation}

Applying the difference operator $d_t$ to \reff{eq4.4} and followed 
by taking the test function $\varphi_h=\tp_h^{n+1}=p_h^{n+1}-\alpha q_h^n$ yield
\begin{align}\label{e3.10}
\bigl(\div d_t\bu_h^{n+1},\tp_h^{n+1} \bigr)
&=\bigl( d_t q_h^n, p_h^{n+1} \bigr) - \alpha \bigl( d_t q_h^n,q_h^n\bigr)
\\
&=\bigl( d_t q_h^n, p_h^{n+1} \bigr)
-\frac{\alpha}2 \Bigl[d_t\norm{q_h^n}{L^2}^2
  +\Del t \norm{d_t q_h^n}{L^2}^2 \Bigr] \nonumber 
\end{align}

To get an expression for the first term on the right-hand side of
\reff{e3.10}, we set $\psi_h=p_h^{n+1}$ in \reff{eq4.6} (after lowering
the index from $n+1$ to $n$ in the equation)
and using the definition $p_h^{n+1}:= \tp_h^{n+1}+\alpha q_h^n$  to get
\begin{align}\label{e3.11}
\bigl(d_t q_h^n, p_h^{n+1} \bigr)
&= -\kappa \bigl(\nab [\tp_h^n+\alpha q_h^n], \nab p_h^{n+1} \bigr) \\
&= -\kappa \norm{\nab p_h^{n+1}}{L^2}^2 
+ \kappa \Del t \bigl(d_t\nab \tp_h^{n+1}, \nab p_h^{n+1} \bigr). \nonumber 
\end{align}

Combining \reff{e3.9}--\reff{e3.11} we obtain
\begin{align}\label{e3.12}
&\frac12 d_t \Bigl[ \beta \norm{\nab \bu_h^{n+1}}{L^2}^2
+ \alpha \norm{q_h^n}{L^2}^2
-2\langle \bbf, \bu_h^{n+1}\rangle_{\small\rm dual} \Bigr]
 + \kappa \norm{\nab p_h^{n+1}}{L^2}^2  \\ 
&\qquad\qquad 
+\frac{\Del t}{2} \Bigl[ \beta \norm{d_t \nab\bu_h^{n+1}}{L^2}^2
 + \alpha \norm{d_t q_h^n}{L^2}^2 \Bigr]
= -\kappa \Del t \bigl(d_t\nab \tp_h^{n+1}, \nab p_h^{n+1} \bigr) \bigr). 
\nonumber
\end{align}

Applying the summation operator
$\Del t \sum_{n=0}^{\ell}$ for any $0\leq \ell \leq \frac{T}{\Del t}-1$ yields
\begin{align}\label{e3.12a}
&\mathcal{J}_h^{\ell+1}
+\Del t\sum_{n=0}^\ell\Bigl[\kappa\norm{\nab p_h^{n+1}}{L^2}^2
+\frac{\Del t}{2} \bigl( \beta \norm{d_t \nab\bu_h^{n+1}}{L^2}^2
+\alpha \norm{d_t q_h^{n+1}}{L^2}^2 \bigr) \Bigr] \\
&\hskip 2in
=\mathcal{J}_h^0-\kappa (\Del t)^2\sum_{n=0}^\ell \bigl( d_t\nab \tp_h^{n+1},
\nab p_h^{n+1} \bigr), \nonumber
\end{align}
where $\mathcal{J}_h^\ell$ is defined by \reff{e3.14}. Hence, \reff{e3.13}
holds. The proof is completed.
\end{proof}

\begin{remark}
The extra term on the right-hand side of \reff{e3.13} is
due to the time lag in \reff{eq4.4} which is needed to decouple the
original whole system into two sub-systems. This term
is smaller if high order time-stepping schemes
are used to replace the implicit Euler scheme in Algorithm $1$.
\end{remark}

\begin{theorem}\label{thm3.3} 
Let $\{(\bu_h^n,\tp_h^n,q_h^n)\}_{n\geq 1}$ be same as in Lemma \ref{lemma3.1}.
Set $q_h^{-1}:\equiv q_h^{0}$. Then there holds the following discrete 
energy inequality: 
\begin{align}\label{e3.100}
&\mathcal{J}_h^{\ell+1} 
+ \Del t\sum_{n=0}^\ell \Bigl[ \frac{\kappa}2 \norm{\nab p_h^{n+1}}{L^2}^2
+\frac{\Del t}{4} \bigl( \beta \norm{d_t \nab\bu_h^{n+1}}{L^2}^2
+\alpha \norm{d_t q_h^{n}}{L^2}^2 \bigr) \Bigr] 
\leq \mathcal{J}_h^0,
\end{align}
provided that $\Del t= O(h^2)$.

\end{theorem}

\begin{proof}
By Schwarz inequality and the inverse inequality we have
\begin{align}\label{e3.101}
\bigl| \bigl( d_t\nab \tp_h^{n+1}, \nab p_h^{n+1} \bigr) \bigr|
&=\frac{1}{\Del t} \Bigl[ \norm{\nab (\tp_h^{n+1}-\tp_h^n)}{L^2}^2
+ \frac{1}2 \norm{\nab p_h^{n+1}}{L^2}^2 \Bigr]\\
&\leq \frac{1}{\Del t} \Bigl[ c_0^2 h^{-2} \norm{\tp_h^{n+1}-\tp_h^n}{L^2}^2
+ \frac{1}2 \norm{\nab p_h^{n+1}}{L^2}^2 \Bigr]. \no
\end{align}
To bound the first term on the right-hand side, we appeal to the 
inf-sup condition \reff{e4.1} (note that $\tp_h^{n+1}-\tp_h^n\in L^2_0(\Ome)$) 
and \reff{eq4.3} to get
\begin{align}\label{e3.102}
\norm{\tp_h^{n+1}-\tp_h^n}{L^2} &\leq \beta_0^{-1} \sup_{v_h\in \bV_h\cap \bX}
\frac{ (\div \bv_h, \tp_h^{n+1}-\tp_h^n)}{\norm{\nab \bv_h}{L^2}} \\
&= \beta_0^{-1} \sup_{v_h\in \bV_h\cap \bX}
\frac{\beta \bigl(\nab(\bu_h^{n+1}-\bu_h^n),\nab\bv_h\bigr)}{\norm{\nab \bv_h}{L^2}} \no \\
&\leq \beta_0^{-1} \beta \norm{ \nab (\bu_h^{n+1}-\bu_h^n)}{L^2}\no\\
&=\beta_0^{-1} \beta \Del t \norm{ d_t\nab \bu_h^{n+1}}{L^2}. \no
\end{align}

\reff{e3.100} follows from combining \reff{e3.13}, \reff{e3.101}, and 
\reff{e3.102} and choosing $\Del t\leq \frac{\beta_0^2}{4\kappa\beta c_0^2} h^2$.
The proof is complete.
\end{proof}

\begin{remark}
It can be shown that all the results obtained in this subsection
for Algorithm 1 also hold for Algorithm 2. The main
differences are (a) the ``correct" discrete pressure $p_h^n$ for Algorithm 2
is $p_h^n:=\tp_h^n+\alpha q_h^n$ instead of $p_h^n:=\tp_h^n+\alpha q_h^{n-1}$;
(b) the ``correct" energy functional $\mathcal{J}_h^\ell$ for Algorithm 2 is
\[
\mathcal{J}_h^{\ell}:= \frac12 \Bigl[ \beta \norm{\nab \bu_h^{\ell}}{L^2}^2
+ \alpha \norm{q_h^{\ell}}{L^2}^2 
- 2\langle \bbf, \bu_h^{\ell}\rangle_{\small\rm dual} \Bigr],
\qquad\ell=0,1,2,\cdots
\]
\end{remark}

\subsection{Convergence analysis}\label{sec-4.3}
In this subsection we derive error estimates for the fully discrete Algorithm $1$.
To the end, we only need to derive the time discretization errors by comparing
the solution of \reff{eq4.1}--\eqref{eq4.7} with that of \reff{e4.3}--\reff{e4.2b}
because the spatial discretization errors have already been derived in
the previous section.  

Introduce notations
\begin{alignat*}{2}
&\bE_{\bu}^n:=\bu_h(t_n)-\bu_h^n,\qquad
&&E_{q}^n:=q_h(t_n)-q_h^n,\\
&E_{\tp}^n:=\tp_h(t_n)-\tp_h^n, \qquad
&&E_p^n:=p_h(t_n)-p_h^n.
\end{alignat*}
It is easy to check that
\[
E_p^n=\hE_p^n + \alpha \Del t d_tq_h(t_n),\qquad
\hE_p^n:=E_{\tp}^n + \alpha E_q^{n-1}. 
\]

\begin{lemma}\label{lem3.2}
Let $\{(\bu_h^n,\tp_h^n,q_h^n)\}_{n\geq 0}$ be generated by the fully
discrete Algorithm $1$ and $\bE_{\bu}^n, E_{q}^n, E_{\tp}^n$, and $E_p^n$  
be defined as above. Set $q_h(t_{-1}):\equiv q_h(t_0)$ and
$q_h^{-1}:\equiv q_h^{0}$. Then there holds the following identity:
\begin{align}\label{e3.19}
&\mathcal{E}_h^{\ell+1}
+\Del t\sum_{n=0}^\ell \Bigl[ \kappa \norm{\nab \hE_p^{n+1}}{L^2}^2
+\frac{\Del t}{2} \bigl( \beta \norm{d_t \nab \bE_{\bu}^{n+1}}{L^2}^2
  + \alpha \norm{d_t E_q^{n}}{L^2}^2 \bigr) \Bigr] \\
&\quad
=\mathcal{E}_h^0-(\Del t)^2
\sum_{n=0}^\ell\Bigl[\kappa \bigl( d_t \nab E_{\tp}^{n+1}, \nab \hE_p^{n+1}\bigr)
- \bigl(d_t^2 q_h(t_{n+1}), E_{\tp}^{n+1} \bigr) \Bigr] \no\\  
&\qquad\qquad
+\Del t\sum_{n=0}^\ell \bigl(R_h^{n},\hE_p^{n+1} \bigr), \no 
\end{align}
where
\begin{align}\label{e3.20}
\mathcal{E}_h^\ell &:= \frac12 \Bigl[ \beta \norm{\nab \bE_{\bu}^{\ell+1}}{L^2}^2
+\alpha \norm{E_q^{\ell}}{L^2}^2  \Bigr],\quad \ell=0,1,2,\cdots, \\
R_h^{n} &:=-\frac{1}{\Del t} \int_{t_{n-1}}^{t_{n}} (s-t_{n-1})(q_h)_{tt}(s)\,ds.
\label{e3.20a}
\end{align}
\end{lemma}

\begin{proof}
First, by Taylor's formula and \reff{e4.2} we have
\begin{align} \label{eq4.100}
\bigl(d_t q_h(t_{n+1}), \psi_h \bigr)
+\kappa \bigl(\nab (\tp_h(t_{n+1}) + \alpha q_h(t_{n+1})) , \nab\psi_h \bigr)
= \bigl( R_h^{n+1}, \psi_h\bigr) \qquad\forall \psi_h \in W_h, 
\end{align}

Then, subtracting \reff{eq4.3} from \reff{e4.3}, \reff{eq4.4} from
\reff{e4.4}, \reff{eq4.5} from \reff{eq3.13},  and \reff{eq4.6} from
\reff{eq4.100}, \reff{eq4.2} from \reff{e4.5a} respectively, we get the 
following error equations:
\begin{alignat}{2} \label{e3.15}
&\beta \bigl( \nab \bE_{\bu}^{n+1},\nab \bv_h\bigr)
-\bigl( E_{\tp}^{n+1},\div \bv_h\bigr) =0 &&\qquad\forall \bv_h\in \bV_h, \\
&\bigl(\div \bE_{\bu}^{n+1},\varphi_h \bigr)
=\bigl(E_q^n,\varphi_h \bigr)+ \Del t\bigl(d_t q_h(t_{n+1}),\varphi_h \bigr)
&&\qquad\forall \varphi_h M_h,  \label{e3.16} \\
&\langle \bE_{\bu}^{n+1}, \nu\rangle =0, \quad \bigl(E_{\tp}^{n+1}, 1\bigr)=0, 
\quad \bE_\bu^0= 0, && \label{e3.16a}\\ 
&\bigl(d_t E_q^{n+1},\psi_h \bigr)+ \kappa \bigl(\nab (E_{\tp}^{n+1}+\alpha E_q^{n+1}),
\nab \psi_h \bigr) = \bigl(R_h^{n+1},\psi_h \bigr)
&&\qquad\forall \psi_h \in W_h,  \label{e3.17} \\
& \bigl(E_{q}^{n+1}, 1\bigr)= 0,\quad E_q^0=0.  &&  \label{e3.18}
\end{alignat}

\reff{e3.19} then follows from setting $\bv_h=d_t E_{\bu}^{n+1}$ in \reff{e3.15},
$\varphi_h= E_{\tp}^{n+1}$ in \reff{e3.16} (after applying the difference
operator $d_t$ to the equation), $\psi_h= \hE_p^{n+1}=E_{\tp}^{n+1} +\alpha E_q^n$ 
in \reff{e3.17} (after lowering the index from $n+1$ to $n$ in the equation), adding
the resulted equations, and applying the summation operator 
$\Del t\sum_{n=0}^\ell$ to the sum. The proof is complete.
\end{proof}

From the above lemma we then obtain the following error estimate.

\begin{theorem}\label{thm3.2}
Let $N$ be a (large) positive integer and $\Del t:=\frac{T}{N}$. Let
the error functions $\bE_{\bu}^n, E_{q}^n, E_{\tp}^n$, and $\hE_p^n$
be same as in Lemma \ref{lem3.2}. Assume $\Del t$ satisfies
the mesh constraint $\Del t < \frac{\beta_0^2}{8\kappa\beta c_0^2} h^2$.
Then there holds error estimate
\begin{align} \label{e3.21}
&\max_{0\leq n\leq  N} \Bigl[ \sqrt{\beta}\norm{\nab \bE_{\bu}^{n}}{L^2}
+\sqrt{\alpha} \norm{ E_q^{n}}{L^2} \Bigr]
+\Bigl( \Del t\sum_{n=0}^{N-1}
\kappa \norm{\nab \hE_p^{n+1}}{L^2}^2 \Bigr)^{\frac12} \\
&\quad 
+\Bigl[\sum_{n=1}^{N-1} 
\bigl(\beta \norm{\nab(\bE_{\bu}^{n+1}-\bE_{\bu}^{n})}{L^2}^2 
+\alpha\norm{E_q^{n}-E_q^{n-1}}{L^2}^2 \bigr) \Bigr]^{\frac12}
\leq C_3(T;q_h)^{\frac12} \Del t, \no
\end{align}
where
\[
C_3(T;q_h):= 16\bigl[ \beta_0^{-2} \beta \norm{ (q_h)_t}{L^2(L^2)}^2 
+ \kappa^{-1} \norm{ (q_h)_{tt}}{L^2(H^{-1})}^2 \bigr].
\]
\end{theorem}

\begin{proof}
To derive the desired error bound, we need to bound each term 
on the right-hand side of \reff{e3.19}. Before doing that, 
on noting that $E_q^0=0$ and $\bE_{\bu}^0=0$ and \reff{e3.20} 
we rewrite \reff{e3.19} as
\begin{align}\label{e3.22}
&\mathcal{E}_h^{\ell+1}
+\Del t \kappa \norm{\nab E_{\tp}^{1}}{L^2}^2
+\frac{\alpha}2 \norm{E_q^{1}}{L^2}^2 \\
&\qquad\qquad
+\Del t\sum_{n=1}^\ell \Bigl[ \kappa \norm{\nab \hE_p^{n+1}}{L^2}^2
+\frac{\Del t}{2} \bigl( \beta \norm{d_t \nab \bE_{\bu}^{n+1}}{L^2}^2
+\alpha \norm{d_t E_q^{n}}{L^2}^2 \bigr) \Bigr] \no \\
&\qquad
=-(\Del t)^2 \sum_{n=0}^\ell\Bigl[\kappa \bigl( d_t \nab E_{\tp}^{n+1}, 
\nab \hE_p^{n+1}\bigr)
-\bigl(d_t^2 q_h(t_{n+1}), E_{\tp}^{n+1} \bigr) \Bigr] \no\\  
&\qquad\qquad
+\Del t\sum_{n=0}^\ell \bigl(R_h^{n},\hE_p^{n+1} \bigr). \no 
\end{align}

We now estimate each term on the right-hand side of \reff{e3.22}. First, 
by Schwarz inequality, the inverse inequality \reff{e4.0}, 
the inf-sup condition \reff{e4.1}, and \reff{e3.15} we have
\begin{align}\label{e3.23}
\kappa \bigl|\bigl( d_t \nab E_{\tp}^{n+1}, \nab \hE_p^{n+1}\bigr)\bigr|
&\leq \kappa c_0 h^{-1} \norm{\nab \hE_p^{n+1}}{L^2}
\norm{d_t E_{\tp}^{n+1}}{L^2} \\
&\leq \kappa c_0 h^{-1} \beta_0^{-1}
\norm{\nab \hE_p^{n+1}}{L^2}\, \sup_{\bv_h\in \bV_h} 
\frac{\bigl(\div\bv_h, d_t E_{\tp}^{n+1}\bigr)}{\norm{\nab\bv_h}{L^2}} \no \\
&\leq \kappa c_0 h^{-1}  \beta_0^{-1} \norm{\nab \hE_p^{n+1}}{L^2} 
\,\sup_{\bv_h\in \bV_h} \frac{\beta\bigl(d_t\nab\bE_{\bu}^{n+1},\nab\bv_h\bigr)}{\norm{\nab\bv_h}{L^2}} \no \\
&\leq \kappa c_0 h^{-1} \beta_0^{-1} \beta \norm{\nab \hE_p^{n+1}}{L^2}  
\norm{d_t\nab\bE_{\bu}^{n+1}}{L^2} \no \\
&\leq \frac{\kappa}{4\Del t} \norm{\nab \hE_p^{n+1}}{L^2}^2 
+\kappa c_0^2 \beta^2 h^{-2} \beta_0^{-2}\Del t \norm{d_t\nab\bE_{\bu}^{n+1}}{L^2}^2.\no
\end{align}   

To bound the second term on the right-hand side of \reff{e3.22}, we first
use the summation by parts formula and $d_t q_h(t_0)=0$ to get
\begin{align}\label{e3.24}
\sum_{n=0}^\ell \bigl(d_t^2 q_h(t_{n+1}), &E_{\tp}^{n+1} \bigr) 
=\frac{1}{\Del t} \bigl(d_t q_h(t_{\ell+1}), E_{\tp}^{\ell+1} \bigr)
-\sum_{n=1}^\ell \bigl( d_t q_h(t_n), d_t E_{\tp}^{n+1}\bigr).
\end{align}
We then bound each term on the right-hand side of \reff{e3.24} as follows:
\begin{align}\label{e3.25} 
\frac{1}{\Del t}\bigl|\bigl(d_t q_h(t_{\ell+1}),E_{\tp}^{\ell+1} \bigr)\bigr|
&\leq \frac{\beta_0^{-1}\beta}{\Del t} \norm{(q_h)_t}{L^2((t_\ell,t_{\ell+1});L^2)} 
\norm{ \nab \bE_{\bu}^{\ell+1}}{L^2} \\
&\leq \frac{\beta}{4 (\Del t)^2} \norm{\nab \bE_{\bu}^{\ell+1}}{L^2}^2 
+\beta_0^{-2} \beta \norm{(q_h)_t}{L^2((t_\ell,t_{\ell+1});L^2)}^2. \no  
\end{align}
\begin{align}\label{e3.27}
\Bigl|\sum_{n=1}^\ell \bigl(d_t q_h(t_n), d_t E_{\tp}^{n+1}\bigr)\Bigl|
&\leq \sum_{n=1}^\ell \norm{d_t q_h(t_n)}{L^2} \norm{d_t E_{\tp}^{n+1}}{L^2} \\ 
&\leq \beta_0^{-1}\beta \sum_{n=1}^\ell 
\norm{d_t q_h(t_n)}{L^2} \norm{d_t \nab\bE_{\bu}^{n+1}}{L^2} \no \\
&\leq  \frac{\beta}{4}\sum_{n=1}^\ell \norm{d_t\nab\bE_{\bu}^{n+1}}{L^2}^2
+ \beta_0^{-2}\beta \norm{(q_h)_t}{L^2(L^2)}^2.  \no
\end{align} 

Finally, we bound the third term on the right-hand side of \reff{e3.22} by
\begin{align}\label{e3.28}
\bigl| \bigl( R_h^n,\hE_p^{n+1} \bigr) \bigr| 
&\leq \norm{R_h^n}{H^{-1}} \norm{\nab \hE_p^{n+1}}{L^2} \\
&\leq \frac{\kappa}4 \norm{\nab \hE_p^{n+1}}{L^2}^2 
+ \frac{\Del t}{\kappa} \norm{(q_h)_{tt}}{L^2((t_{n-1},t_n);H^{-1})}^2,\no  
\end{align}
where we have used the fact that
\[
\norm{R_h^n}{H^{-1}}^2 \leq \frac{\Del t}3 \int_{t_{n-1}}^{t_n}
\norm{ (q_h)_{tt}(t)}{H^{-1}}^2 \, dt.
\]

Substituting \reff{e3.23}--\reff{e3.28} into \reff{e3.22} and using the 
assumption that $\Del t < \frac{\beta_0^2}{8\kappa\beta c_0^2} h^2$ yields
\begin{align*}
&\beta \norm{\bE_{\bu}^{\ell+1}}{L^2}^2 
+\alpha\norm{E_q^\ell}{L^2}^2  
+\Del t \bigl[ \kappa \norm{\nab E_{\tp}^{1}}{L^2}^2
+\frac{\alpha}2 \norm{E_q^{1}}{L^2}^2 \bigr] \\
&\quad
+\Del t \sum_{n=1}^\ell \Bigl[ \kappa \norm{\nab \hE_p^{n+1}}{L^2}^2
+\frac{\Del t}{2} \bigl( \beta \norm{d_t \nab \bE_{\bu}^{n+1}}{L^2}^2
+\alpha \norm{d_t E_q^{n}}{L^2}^2 \bigr) \Bigr] \no \\
&\hskip 0.8in
\leq 16(\Del t)^2 \bigl[ \beta_0^{-2}\beta \norm{ (q_h)_t}{L^2(L^2)}^2 
+ \kappa^{-1} \norm{ (q_h)_{tt}}{L^2(H^{-1})}^2 \bigr],
\end{align*}
which trivially implies \reff{e3.19}. The proof is complete.
\end{proof}

\begin{theorem}\label{thm4.2}
The solution of the fully discrete Algorithm 1 satisfies the following 
error estimates:
\begin{align} \label{e4.38}
&\max_{0\leq n\leq N}
\bigl[\sqrt{\beta} \norm{\nab(\bu(t_n)-\bu_h^n)}{L^2}
+\sqrt{\alpha} \norm{q(t_n)-q_h^n}{L^2} \bigr] \\
&\qquad
+\Bigl(\Del t\sum_{n=0}^{N}\kappa\norm{\nab(p(t_n)-p_h^n)}{L^2}^2\Bigr)^{\frac12}
\no \\
&\qquad\qquad
\leq \bigl[ C_1(T;\bu,\tp,p,q)^{\frac12}\,e^{\frac{T}{2}}
+\widehat{C}_1(T;\bu,p,q)^{\frac12} \bigr] h
+\widehat{C}_3(T;q_h)^{\frac12} \Del t, \no
\end{align}
provided that $C_1(T;\bu,\tp,p,q)<\infty, \widehat{C}_1(T;\bu,p,q)<\infty$,
$\widehat{C}_3(T;q_h)<\infty$, and 
$\Del t < \frac{\beta_0^2}{8\kappa\beta c_0^2} h^2$.

Moreover, if, in addition, $C_2(T;\bu,\tp,p,q)<\infty$ and 
$\widehat{C}_2(T;\bu,q)<\infty$, then there also holds
\begin{align} \label{e4.39}
&\max_{0\leq n\leq N}
\bigl[\sqrt{\beta} \norm{\nab(\bu(t_n)-\bu_h^n)}{L^2}
+\sqrt{\alpha} \norm{q(t_n)-q_h^n}{L^2} \bigr]\\
&\hskip 1in
\leq \bigl[ C_2(T;\bu,\tp,p,q)^{\frac12}\,e^{\frac{T}{2}}
+\widehat{C}_2(T;\bu,q)^{\frac12} \bigr] h^2
+\widehat{C}_3(T;q_h)^{\frac12} \Del t. \no
\end{align}

\end{theorem}

\begin{proof}
The assertions follow easily from first using the triangle inequality on
\begin{align*}
\bu(t_n)-\bu_h^n &= \bE_{\bu}(t_n) + \bE_{\bu}^n, \\
q(t_n)-q_h^n &=E_q(t_n) + E_q^n, \\
p(t_n)-p_h^n &=E_p(t_n) + E_p^n,
\end{align*}
and then appealing to Theorems \ref{thm3.2} and \ref{thm4.2}.
\end{proof}

\begin{remark}
(a) In light of Theorems \ref{regularity} and \ref{regularity_1}, the 
regularity assumptions of Theorem \ref{thm4.2} are valid if the domain
$\Ome$ and datum functions $\bbf$ and $\bu_0$ are sufficient regular.

(b) It can be shown that all the results proved in this subsection for Algorithm 1 
still hold for Algorithm 2. The main differences are (i) the ``correct" 
discrete pressure $p_h^n$ for Algorithm 2 is $p_h^n:=\tp_h^n+\alpha q_h^n$;
(i) the ``correct" error functional $\mathcal{E}_h^\ell$ for Algorithm 2 is 
\[
\mathcal{E}_h^{\ell}:= \frac12 \Bigl[ \beta \norm{\nab \bE_{\bu}^{\ell}}{L^2}^2
+ \alpha \norm{E_q^{\ell}}{L^2}^2 \Bigr], \qquad\ell=0,1,2,\cdots
\]
\end{remark}

\section{Numerical experiments}\label{sec-5}

In this section we present some $2$-D numerical experiments to gauge the
efficiency of the fully discrete finite element methods developed in this paper.
Three tests are performed on two different geometries. The gel used in
all three tests is the Ploy(N-isopropylacrylamide) (PNIPA) hydrogel
(cf. \cite{titm} and the references therein). The material constants/parameters,
which were reported in \cite{titm}, are given as follows:
\begin{alignat*}{2}
&E=6\times 10^3, &&\qquad \mbox{Young's modulus},\\
&\nu=0.43, &&\qquad \mbox{Poisson's ratio}, \\
&K=\frac{E}{3(1-2\nu)}=14285.7, &&\qquad \mbox{bulk modulus},\\
&G=\frac{E}{2(1+\nu)}=2097.9, &&\qquad \mbox{shear modulus}.
\end{alignat*}
Two other material constants/parameters, which were not given in \cite{titm}, 
are taken as follows in our numerical tests:
\begin{alignat*}{2}
&\varphi=0.15, &&\qquad \mbox{porosity},\\
&\xi=100, &&\qquad \mbox{friction constant}.
\end{alignat*}
In addition, we use the following initial condition in all our numerical
tests:
\[
\bu_0(x)=10^{-4} \sin(x_1+x_2) (1,1).
\]

\bigskip
{\bf Test 1:} Let $\Ome=(0,1)^2$. The external force is taken as
\begin{align*}
\bbf=(f_1,f_2)=0.1\mathbf{t}_{\mbox{\tiny tangent}},
\end{align*}
where $\mathbf{t}_{\mbox{\tiny tangent}}$ denotes the unit (clockwise) 
tangential vector on $\p\Ome$. Note that the compatibility condition
\begin{equation}\label{e5.1}
\int_{\p\Ome} \bbf(x)\, dS=0
\end{equation}
is trivially fulfilled.

\begin{figure}[htb]
\centerline{
\includegraphics[scale=0.2]{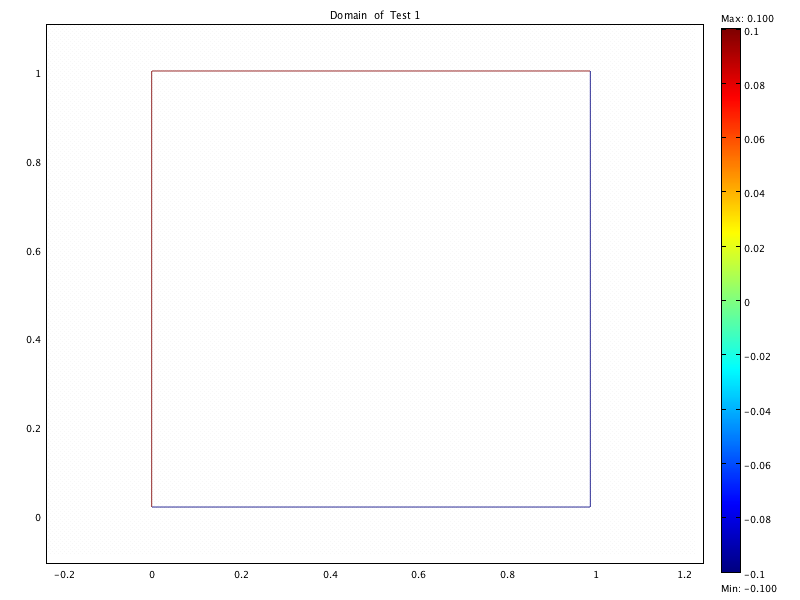}
\includegraphics[scale=0.2]{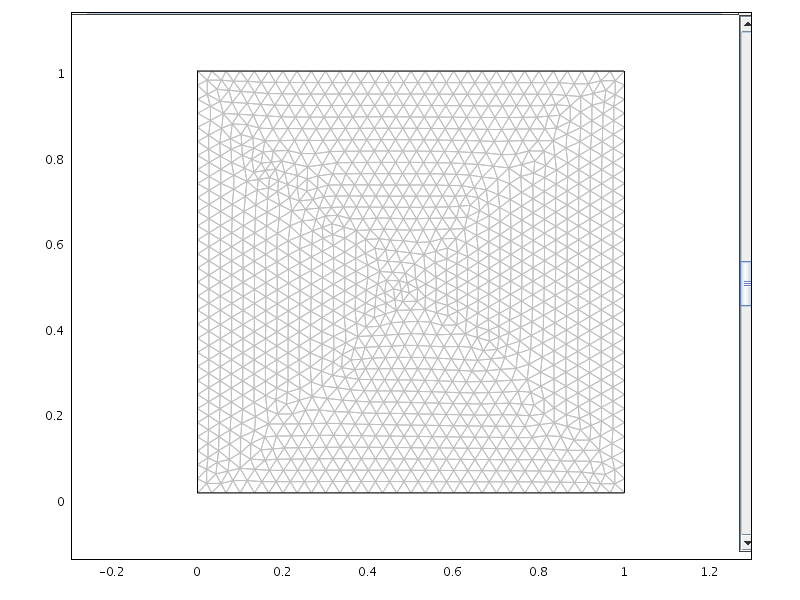}
}
\caption{Computational domain, boundary data, and mesh of Test 1}
\label{fig1a}
\end{figure}

Figure \ref{fig1a} shows the computational domain, color plot of 
the force function $f_1+f_2$, and the mesh on which the numerical 
solution is computed.  The mesh consists of $2360$ elements, and
total number of degrees of freedom for the test is $12164$. 
$\Del t=0.01$ is used in this test. 

Figure \ref{fig1b} displays snapshots of the computed solution  
at three time incidences. Each graph contains color plot of the
computed pressure $p_h^n:=\tp_h^n+\alpha q_h^{n-1}$ and arrow plot 
of the computed displacement field $\bu_h^n$. The three graphs 
on the first row are plotted on the computational domain $\Ome=(0,1)^2$,
while three graphs on the second row are respectively deformed shape plots 
of the three graphs on the first row with $500$ times magnification,
which shows the deformation of the square gel under the mechanical 
force $\bbf$ on the boundary. As expected, the gel is slightly 
rotated clockwise and is slightly bent near the top and bottom edges.
We also note that the expected conserved quantities are indeed conserved
in the computation, their respective values are given as follows:
\begin{alignat*}{2}
&C_{\bu}=\int_{\p\Ome} \bu_h^n(x)\, dS\equiv 9.935\times 10^{-5},\quad
&& C_q=\int_{\Ome} q_h^n(x)\, dx\equiv 9.935\times 10^{-5},\\
&C_{\tp}=\int_{\Ome} \tp_h^n(x)\, dx\equiv 0, \quad
&& C_{p}=\int_{\Ome} p_h^n(x)\, dx\equiv 1.489. 
\end{alignat*}

\begin{figure}[htb]
\centerline{
\includegraphics[scale=0.18]{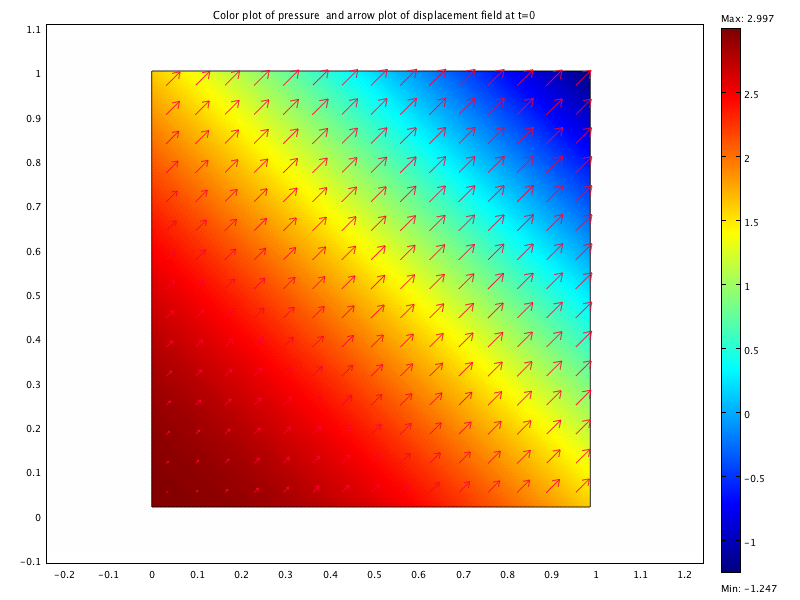}
\includegraphics[scale=0.18]{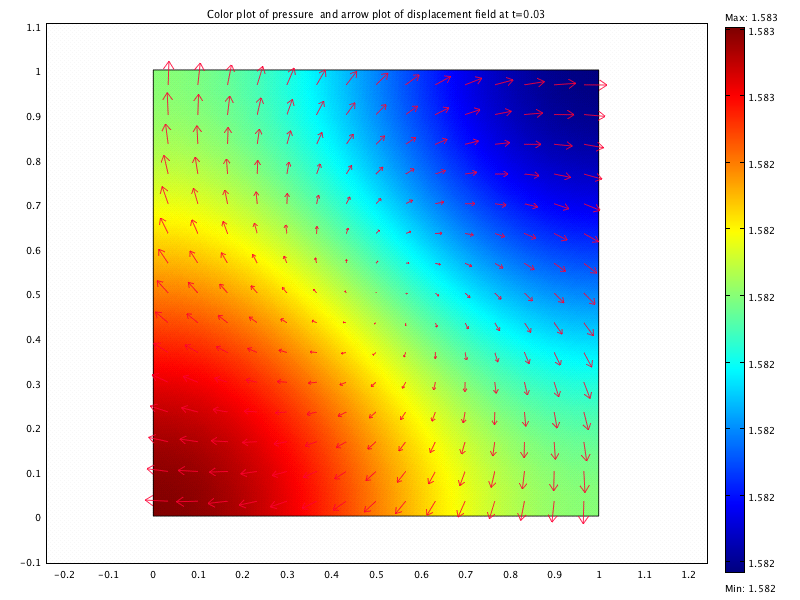}
\includegraphics[scale=0.18]{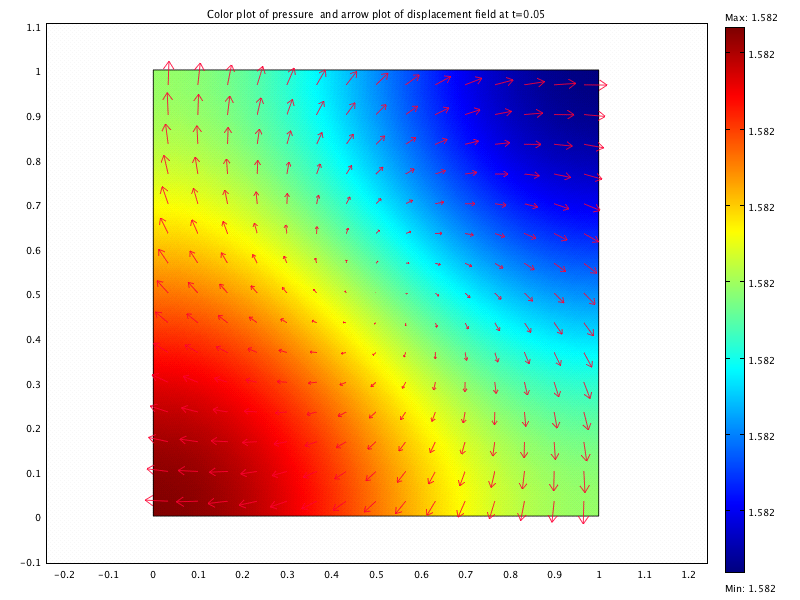}
}
\centerline{
\includegraphics[scale=0.18]{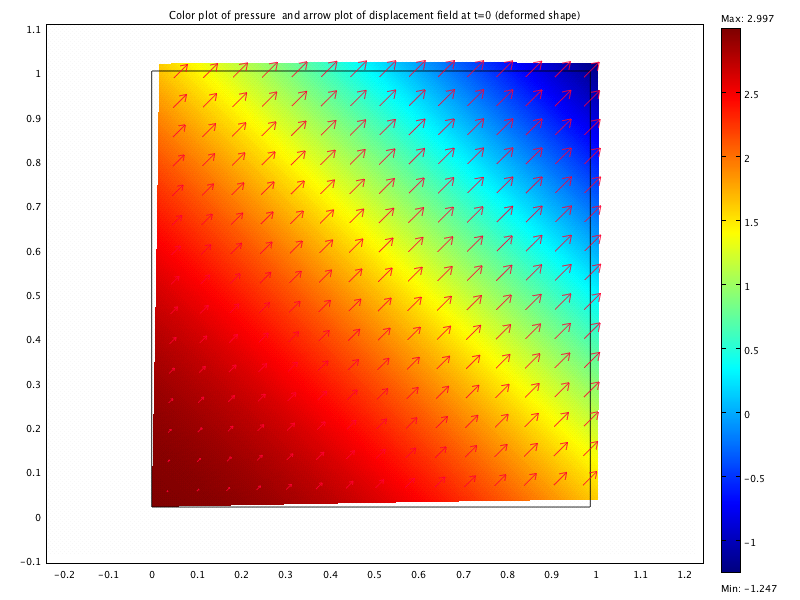}
\includegraphics[scale=0.18]{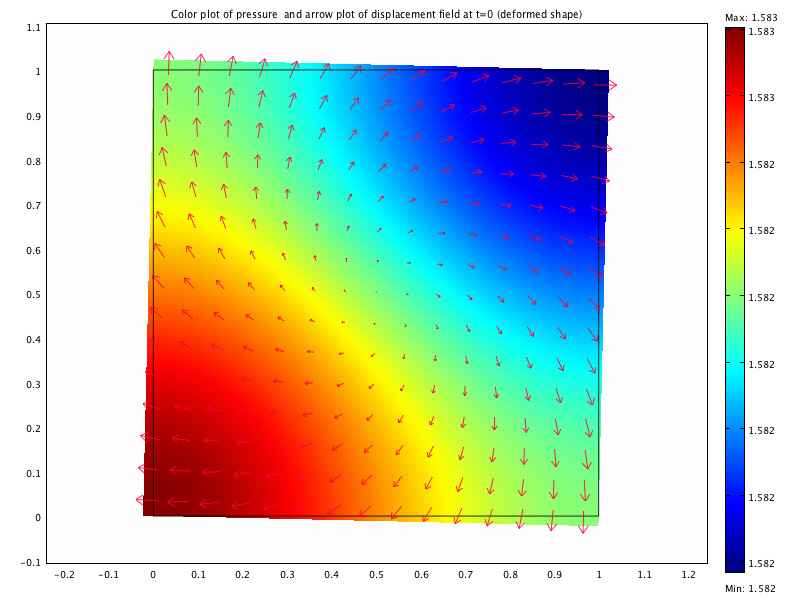}
\includegraphics[scale=0.18]{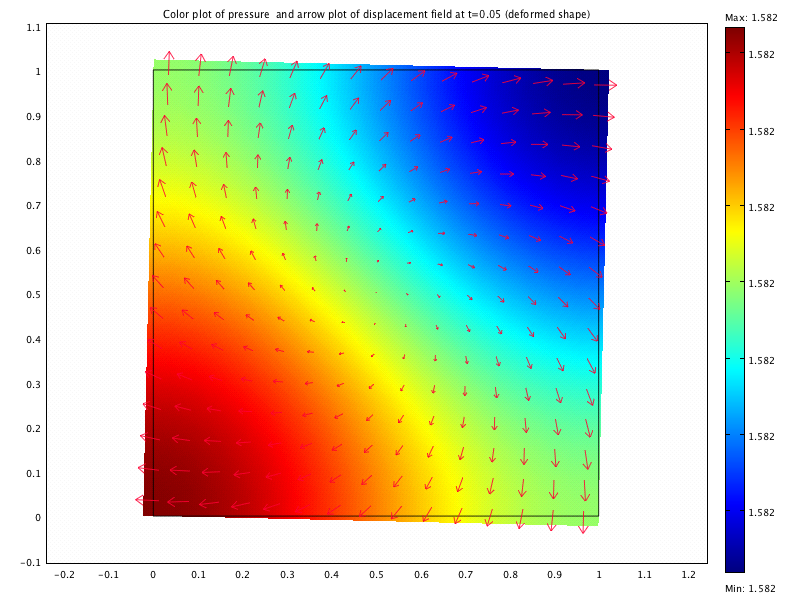}
}
\caption{Test 1: color plots of computed pressure $p_h^n$ and arrow plots of 
computed displacement $\bu_h^n$ at $n=0,3,5$ (first row). Deformed shape plots of 
the graphs on the first row (second row). $\Del t= 0.01$.} 
\label{fig1b}
\end{figure}

\bigskip
{\bf Test 2:} This test is same as Test 1 except $\bbf=(f_1,f_2)$ is changed to
\begin{align*}
&f_1(x):= \left\{ \begin{array}{rl}
           0.5 & \qquad \mbox{for } x_1=0,\, 0\leq x_2\leq 1,\\
           -0.5 & \qquad \mbox{for } x_1=1,\, 0\leq x_2\leq 1, \\
             0 & \qquad \mbox{for } x_2=0,1,\, 0\leq x_1\leq 1,
            \end{array} \right. 
\\
\\
&f_2(x):\equiv 0.
\end{align*}
So parallel forces of opposite directions are
applied at the left and the right boundary of the square gel. 
Clearly, the compatibility condition \reff{e5.1} is satisfied.

\begin{figure}[htb]
\centerline{
\includegraphics[scale=0.2]{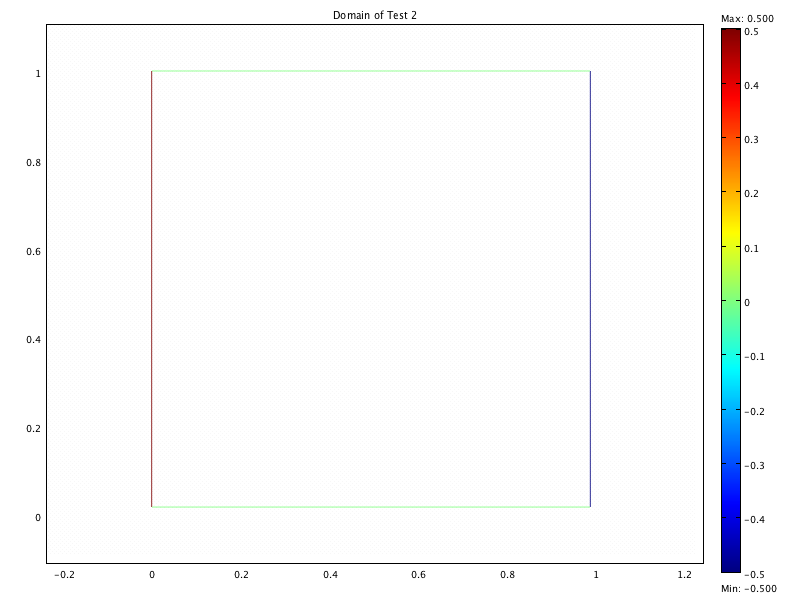}
\includegraphics[scale=0.2]{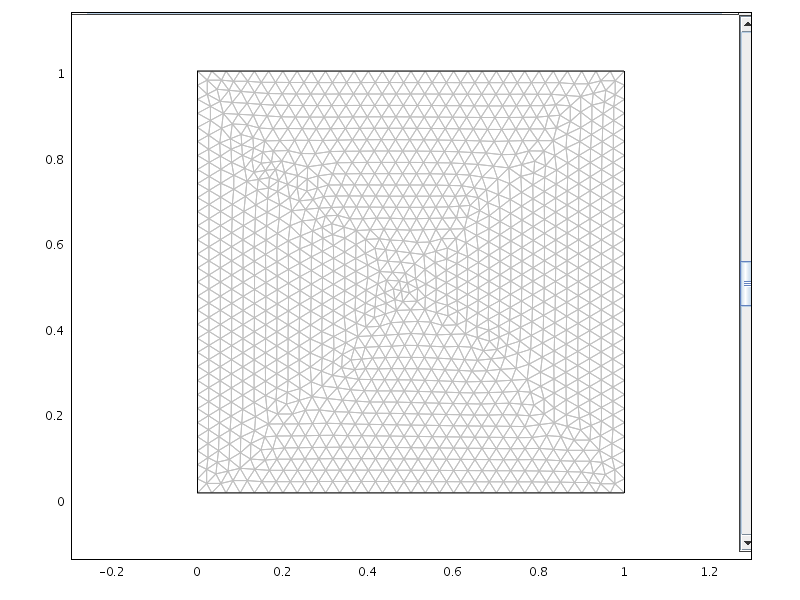}
}
\caption{Computational domain, boundary data, and mesh of Test 2}
\label{fig2a}
\end{figure}

Like Figure \ref{fig1a}, Figure \ref{fig2a} shows the computational domain, 
color plot of the force function $f_1+f_2$. The mesh parameters
are also same as those of Test 1, including the time step $\Del t=0.01$. 

\begin{figure}[htb]
\centerline{
\includegraphics[scale=0.18]{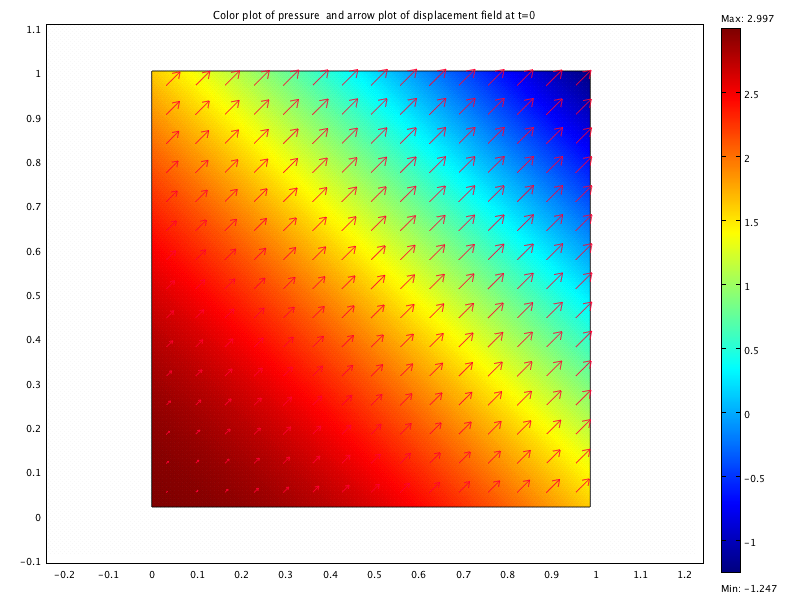}
\includegraphics[scale=0.18]{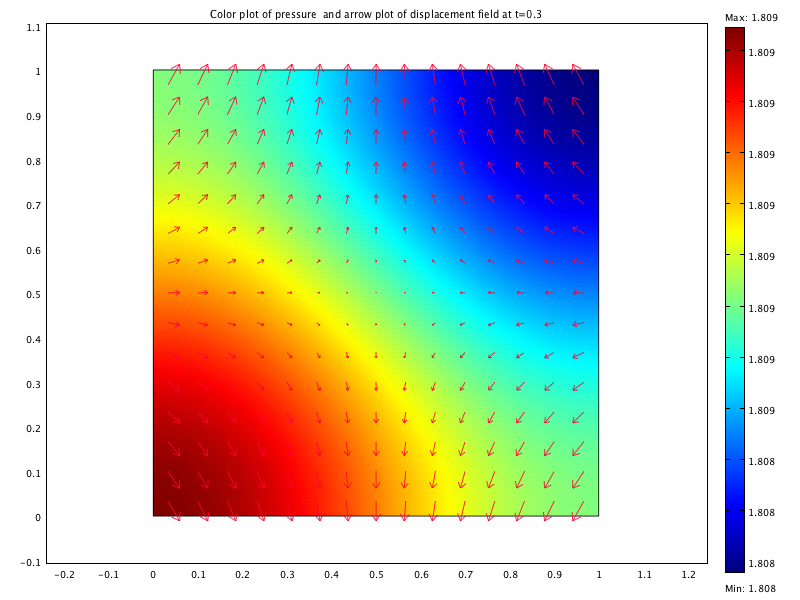}
\includegraphics[scale=0.18]{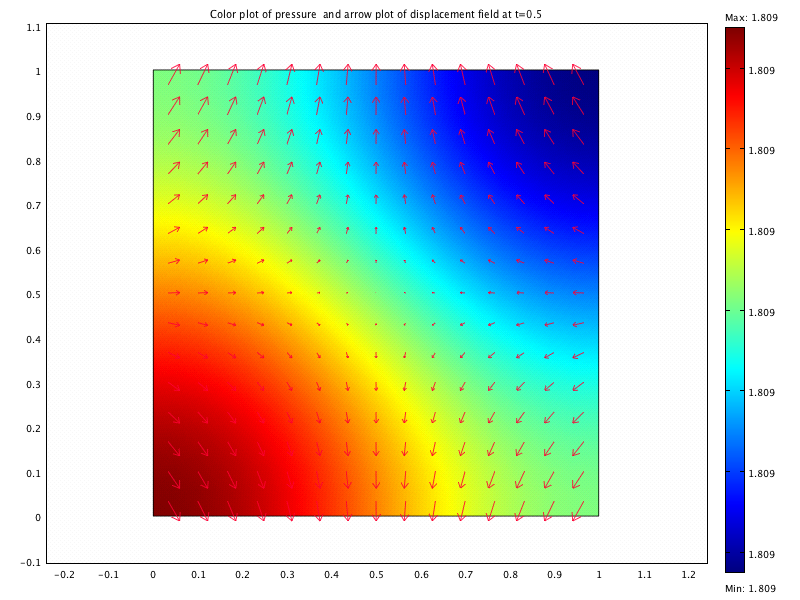}
}
\centerline{
\includegraphics[scale=0.18]{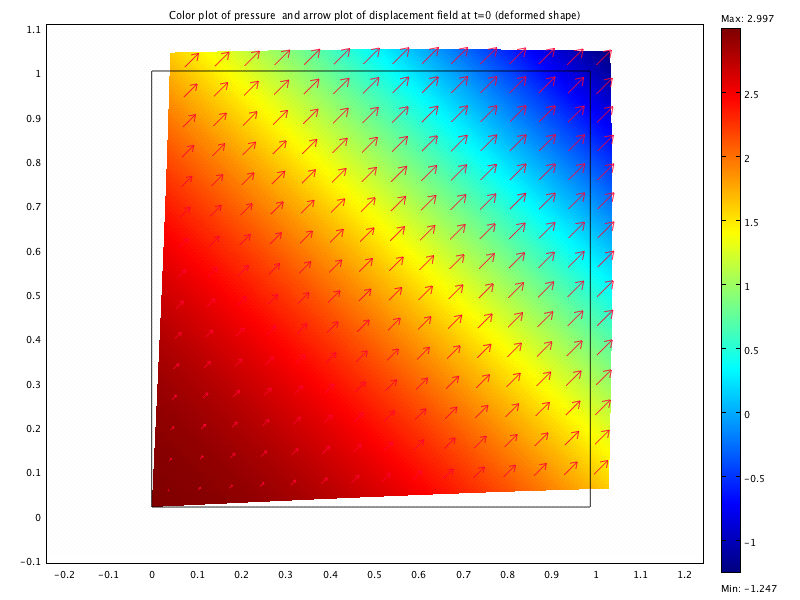}
\includegraphics[scale=0.18]{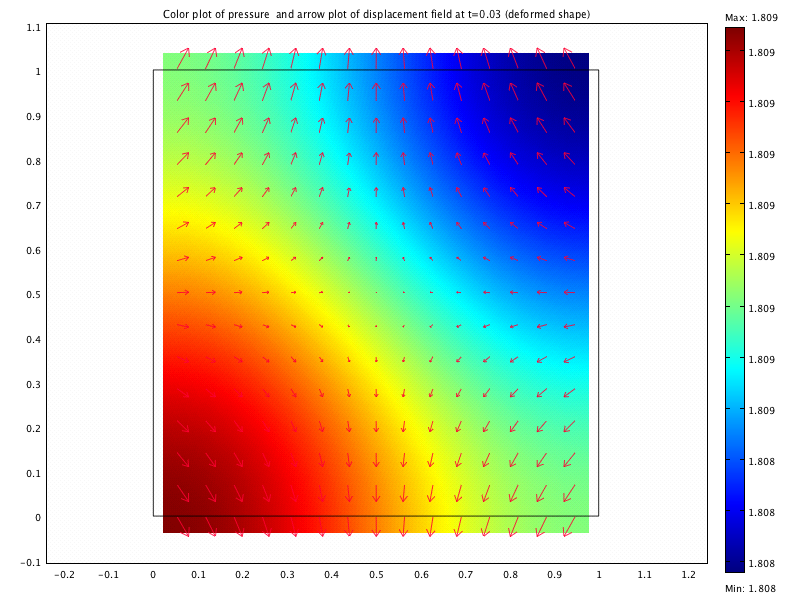}
\includegraphics[scale=0.18]{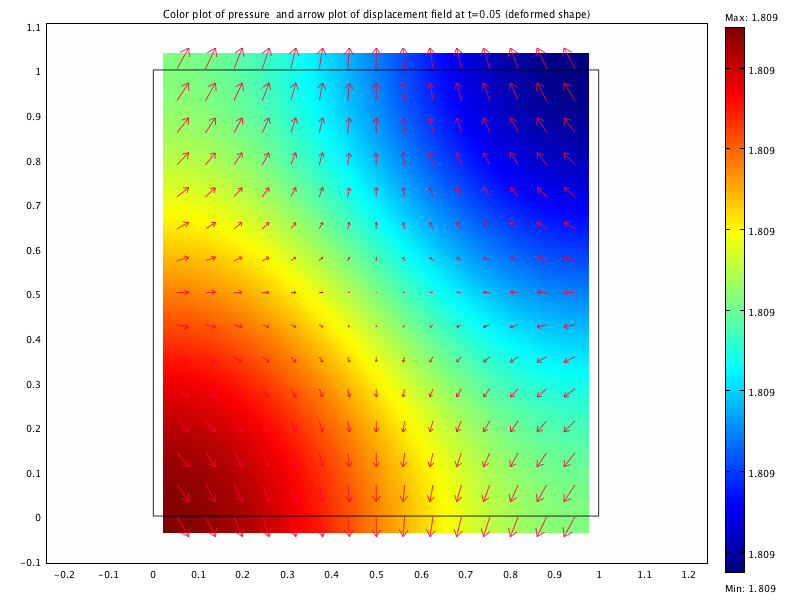}
}
\caption{Test 2: color plots of computed pressure $p_h^n$ and arrow plots of
computed displacement $\bu_h^n$ at $n=0,3,5$ (first row). Deformed shape plots of
the graphs on the first row (second row). $\Del t= 0.01$.}
\label{fig2b}
\end{figure}

Figure \ref{fig2b} is the counterpart of Figure \ref{fig1b} for Test 2.
Since parallel forces of equal magnitude but opposite directions are
applied on the left and the right boundary of the square gel, so
the gel is squeezed horizontally. Because of the incompressibility of
gel, the total volume must be conserved. As a result, the deformation
in the vertical direction is expected. The second row of Figure \ref{fig2b}
precisely shows such a deformation (with $500$ times magnification).
It should be noted that the swelling dynamics of the gel goes super fast,
it reaches the equilibrium in very short time. For the conserved
quantities, $C_q$ and $C_{\bu}$ are same as those in Test 1, and
the other two numbers are given by
\begin{align*}
C_{\tp}=\int_{\Ome} \tp_h^n(x)\, dx\equiv 0.321, \qquad
C_{p}=\int_{\Ome} p_h^n(x)\, dx\equiv 1.809. 
\end{align*}
Recall that $C_{\tp}$ and $C_p$ depend on the force function $\bbf$.

\bigskip
{\bf Test 3:} Same material parameters/constants and initial condition
$\bu_0$ as in Tests 1 and 2 are assumed. However, the computational domain
is changed to the following one
\[
\Ome:=\Bigl\{x\in \mathbf{R}^2;\, \frac{x_1^2}{0.16}+\frac{x_2^2}{0.04}\leq 1\Bigr\},
\]
and the external force function $\bbf$ is taken as
\begin{align*}
&f_1(x):= \left\{ \begin{array}{rl}
           0.5 & \, \mbox{for } x\in\p\Ome, \, -0.2<x_1<0, \\
           -0.5 & \, \mbox{for } x\in\p\Ome, \, 0<x_1<0.2 ,
            \end{array} \right.  
\\
\\
&f_2(x):\equiv 0
\end{align*}
As in Test 2, the above $\bbf$ means that parallel forces of same magnitude but
opposite directions are applied at the left half and the right half ellipse (boundary),
however, these two forces now collide at two points $(0,\pm 0.2)$ on the boundary.
Clearly, the compatibility condition \reff{e5.1} is satisfied.

\begin{figure}[htb]
\centerline{
\includegraphics[scale=0.2]{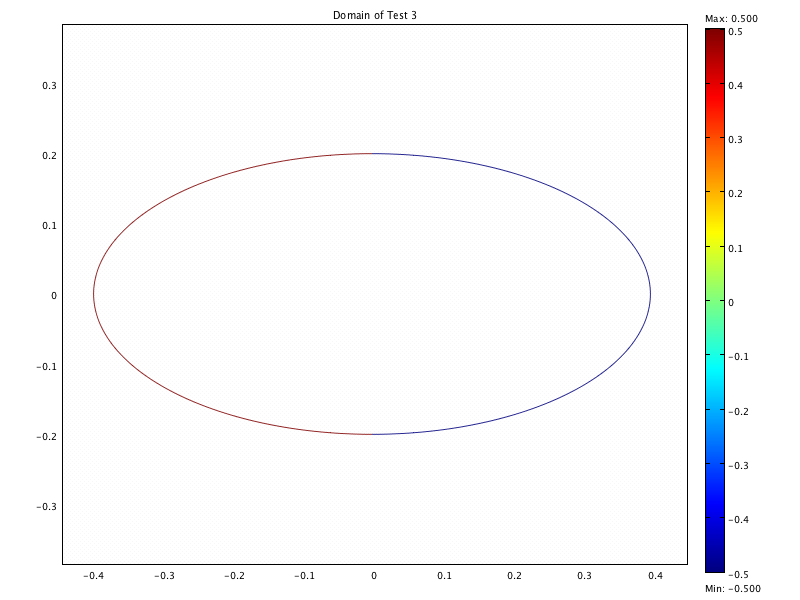}
\includegraphics[scale=0.2]{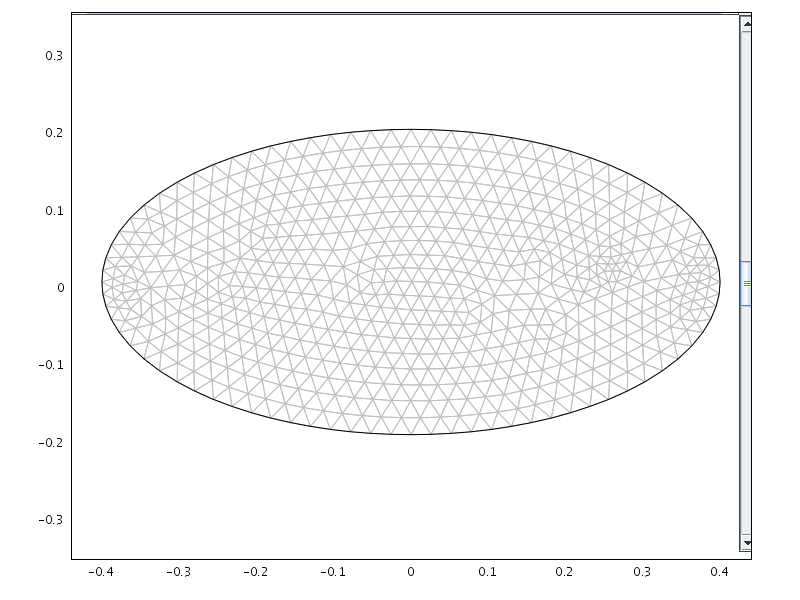}
}
\caption{Computational domain, boundary data, and mesh of Test 3}
\label{fig3a}
\end{figure}

\begin{figure}[htb]
\centerline{
\includegraphics[scale=0.18]{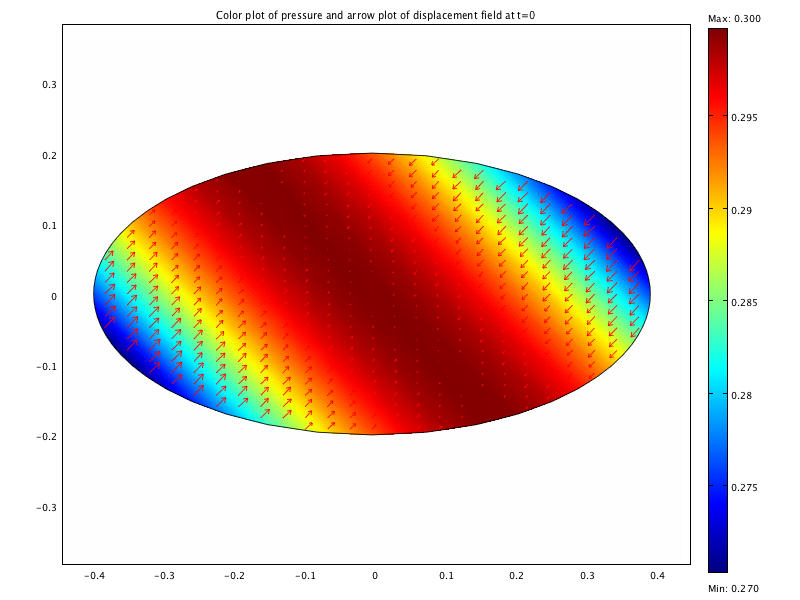}
\includegraphics[scale=0.18]{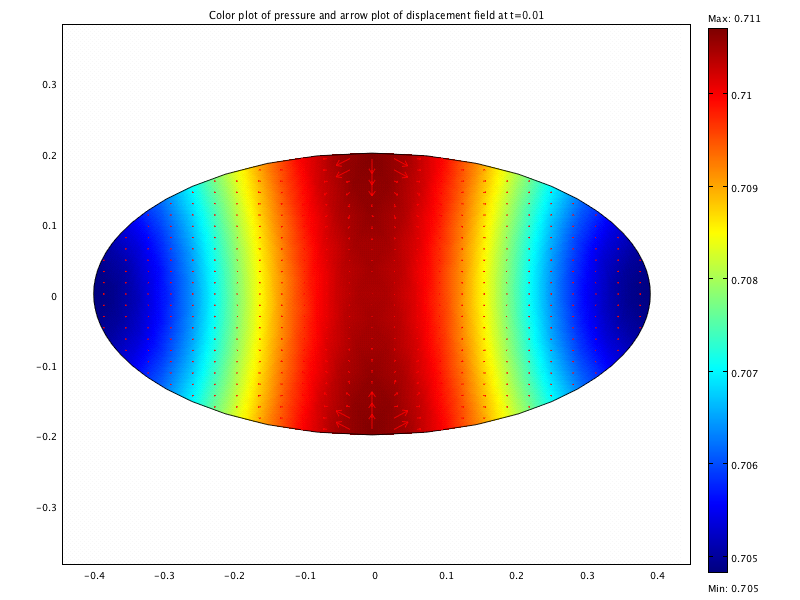}
\includegraphics[scale=0.18]{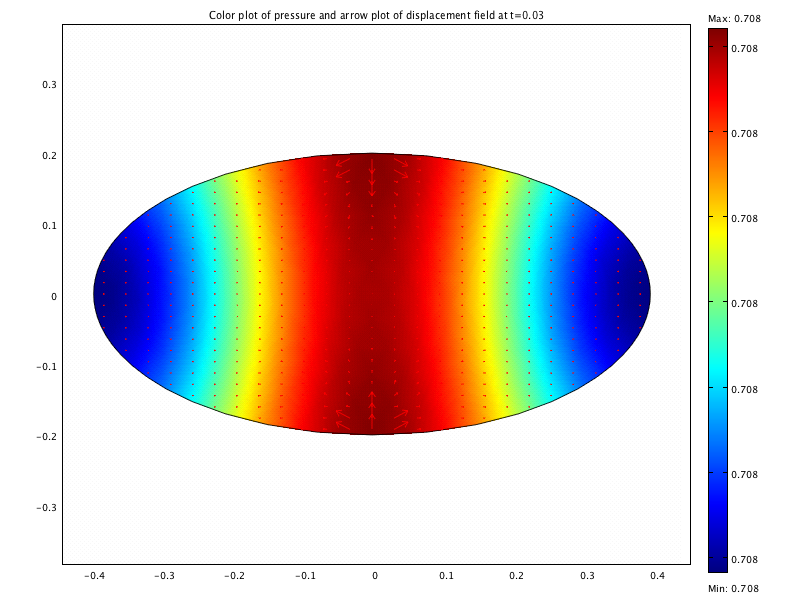}
}
\centerline{
\includegraphics[scale=0.18]{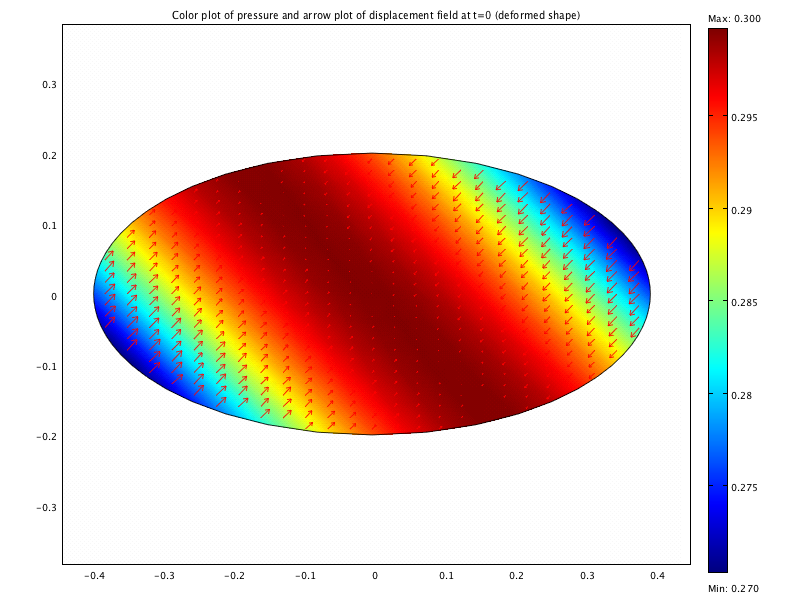}
\includegraphics[scale=0.18]{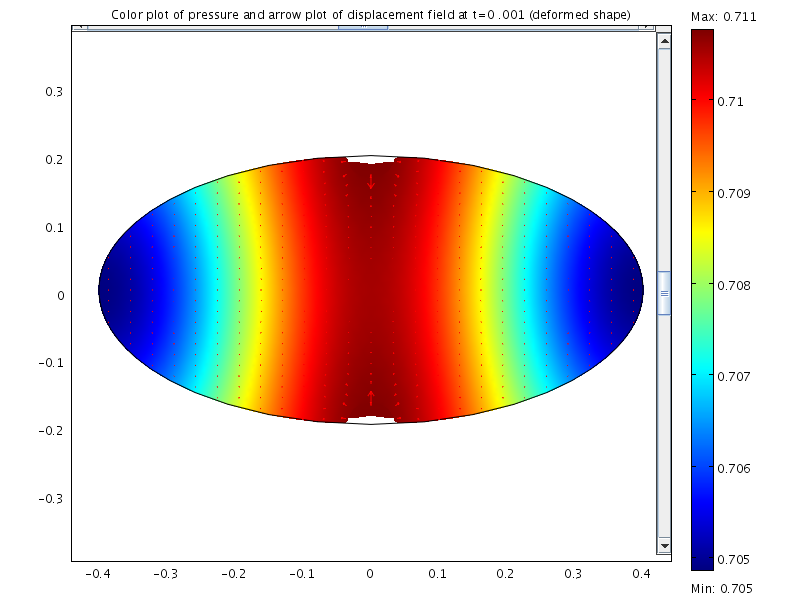}
\includegraphics[scale=0.18]{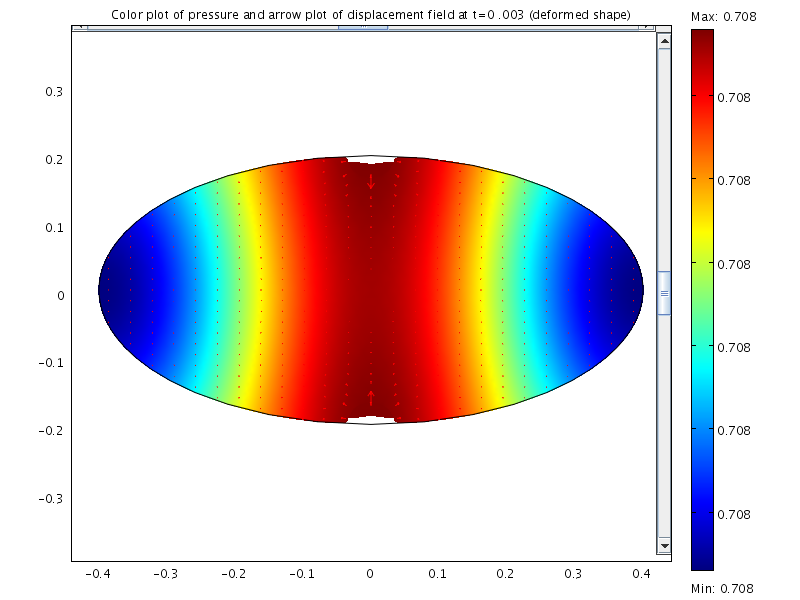}
}
\caption{Test 3: color plots of computed pressure $p_h^n$ and arrow plots of
computed displacement $\bu_h^n$ at $n=0,1,3$ (first row). Deformed shape plots of
the graphs on the first row (second row). $\Del t= 0.001$.}
\label{fig3b}
\end{figure}

Figure \ref{fig3a} is the counterpart of Figures \ref{fig1a} and \ref{fig2a} 
for Test 3. The mesh consists of $1200$ elements, and total number of degrees 
of freedom for the test is $6244$. A smaller time step $\Del t=0.001$ is used 
in this test.  

Figure \ref{fig3b} is the counterpart of Figures \ref{fig1b} and \ref{fig2b} 
for Test 3. The deformations displayed on the second row are magnified by
$20$ times instead of $500$ times as in Figures \ref{fig1b} and \ref{fig2b}.
Due to curve boundary, the applied parallel forces of same magnitude but 
opposite directions collide at the boundary points $(0,\pm 0.2)$.
It is expected that the gel should buckle around those two points,
which is clearly seen in the plots on the second row of
Figure \ref{fig3b}. We also note that because relatively bigger 
forces are applied on the gel, the swelling dynamics of the gel goes 
even faster, hence, reaches the equilibrium quicker. This is the main 
reason to use a smaller time step for the simulation.

The conserved quantities for Test 3 are given as follows:
\begin{alignat*}{2}
&C_{\bu}=\int_{\p\Ome} \bu_h^n(x)\, dS\equiv 4.902\times 10^{-6},
&&\qquad C_q=\int_{\Ome} q_h^n(x)\, dx\equiv 4.902\times 10^{-6},\\
&C_{\tp}=\int_{\Ome} \tp_h^n(x)\, dx\equiv 0.105, 
&&\qquad C_{p}=\int_{\Ome} p_h^n(x)\, dx\equiv 0.178. 
\end{alignat*}

\medskip
{\bf Acknowledgment:}  The first author would like to thank Professor Masao Doi
of Tokyo University for introducing the gel swelling dynamic model to the author,
and for his many stimulating discussions at IMA of University of Minnesota,
where they both were long-term visitors in Fall 2004.


\end{document}